\documentclass[12pt]{article}
\usepackage[final]{showkeys}
\usepackage[obeyspaces,hyphens,spaces]{url}
\renewcommand*\showkeyslabelformat[1]{%
\fbox{\parbox[t]{1.4 cm}{\raggedright\normalfont\small\url{#1}}}}
\usepackage{amsthm,amsmath,amsfonts,mathtools,wasysym,empheq}
\usepackage{amssymb} %comment out if use mtpro2
\usepackage{mathpazo}
\usepackage[mathscr]{eucal}
\usepackage[top=2cm,bottom=2cm,left=2cm,right=2cm]{geometry}
\usepackage[title]{appendix}
\usepackage[table, dvipsnames]{xcolor} %
\usepackage[labelsep=space]{caption}  % or period
\usepackage{graphicx, soul}
\usepackage{float,booktabs,multirow}
\restylefloat{table*}
\definecolor{labelkey}{rgb}{.1,.1,.8}
\definecolor{refkey}{rgb}{0,0.6,0.0}

\usepackage[shortlabels]{enumitem}
\setlist[enumerate]{itemsep=0.5pt,topsep=1pt}
\setlist[itemize]{itemsep=0.5pt,topsep=1pt}
\definecolor{dgreen}{rgb}{0.00,0.49,0.00}
\definecolor{dblue}{rgb}{0,0.08,0.75}
\colorlet{myblue}{dblue}
\colorlet{mygreen}{dgreen}
\definecolor{myfirstblue}{rgb}{.8, .8, 1}

\usepackage{hyperref}
\hypersetup{
colorlinks= true,
urlcolor = magenta,
citecolor = mygreen,
linkcolor = myblue,
}

\allowdisplaybreaks
\usepackage[capitalize]{cleveref} %nameinlink for color in name
\crefdefaultlabelformat{#2\textup{#1}#3}
\crefrangeformat{enumi}{#3\textup{#1}#4\textup{\textendash}#5\textup{#2}#6}
\crefname{equation}{}{}

\crefname{chapter}{Appendix}{chapters}
\crefname{item}{}{items}
\crefname{figure}{Figure}{Figures}
\crefname{theorem}{\protect\theoremname}{Theorems}
\crefname{lemma}{\protect\lemmaname}{Lemmas}
\crefname{proposition}{\protect\propositionname}{Propositions}
\crefname{corollary}{\protect\corollaryname}{\protect\corollaryname}
\crefname{definition}{\protect\definitionname}{\protect\definitionname}
\crefname{fact}{\protect\factname}{\protect\factname}
\crefname{example}{\protect\examplename}{Examples}
\crefname{algorithm}{Algorithm}{Algorithms}
\crefname{remark}{\protect\remarkname}{\protect\remarkname}
\crefname{case}{\protect\casename}{\protect\casename}
\crefname{question}{\protect\questionname}{\protect\questionname}
\crefname{claim}{\protect\claimname}{\protect\claimname}
\crefname{enumi}{}{}
\crefname{appsec}{Appendix}{Appendices}

\makeatletter
\g@addto@macro\normalsize{%
  \setlength\abovedisplayskip{6pt}
  \setlength\belowdisplayskip{6pt}
  \setlength\abovedisplayshortskip{6pt}
  \setlength\belowdisplayshortskip{6pt}
}
\let\orgdescriptionlabel\descriptionlabel
\renewcommand*{\descriptionlabel}[1]{%
	\let\orglabel\label
	\let\label\@gobble
	\phantomsection
	\edef\@currentlabel{#1}%
	\let\label\orglabel
	\orgdescriptionlabel{#1}%
}

\let\leq\leqslant
\let\geq\geqslant

\def\th@plain{%
	\thm@notefont{} % same as heading font
	\itshape % body font
}
\def\th@definition{%
	\thm@notefont{}% same as heading font
	\normalfont % body font
}
\g@addto@macro\th@remark{\thm@headpunct{}}
\g@addto@macro\th@definition{\thm@headpunct{}}
\g@addto@macro\th@plain{\thm@headpunct{}}
\makeatother
\theoremstyle{plain}
\newtheorem{theorem}{\protect\theoremname}[section]
\newtheorem{corollary}[theorem]{\protect\corollaryname}
\newtheorem{lemma}[theorem]{\protect\lemmaname}

\theoremstyle{definition}
\newtheorem{remark}[theorem]{\protect\remarkname}
\newtheorem{definition}[theorem]{\protect\definitionname}
\newtheorem{example}[theorem]{\protect\examplename}

\newtheorem{fact}[theorem]{\protect\factname}

\newtheorem*{openprob}{Open Problem}
\providecommand{\theoremname}{Theorem}
\providecommand{\propositionname}{Proposition}
\providecommand{\corollaryname}{Corollary}
\providecommand{\factname}{Fact}
\providecommand{\lemmaname}{Lemma}
\providecommand{\assumptionname}{Assumption}
\providecommand{\algorithmname}{Algorithm}
\providecommand{\definitionname}{Definition}
\providecommand{\notationname}{Notation}
\providecommand{\remarkname}{Remark}
\providecommand{\examplename}{Example}
\providecommand{\claimname}{Claim}
\providecommand{\algorithmname}{Algorithm}

\let\originalleft\left
\let\originalright\right
\renewcommand{\left}{\mathopen{}\mathclose\bgroup\originalleft}
\renewcommand{\right}{\aftergroup\egroup\originalright}
\DeclarePairedDelimiterX\menge[2]{ \{ }{ \} }{ {#1} ~ \delimsize \vert ~ \mathopen{}  {#2} }  % set	
\DeclarePairedDelimiterX\fa[2]{ ( }{ )_{#2} }{#1}  % indexed family of vectors
\DeclarePairedDelimiterX\set[2]{ \{ }{ \}_{#2} }{#1}  % for indexed set

\newcommand{\scal}[2]{\langle{{#1},{#2}}\rangle}

\newcommand{\barx}{\ensuremath{\overline{x}}}

\newcommand{\HH}{\ensuremath{\mathcal H}}
\newcommand{\RR}{\ensuremath{\mathbb R}}
\newcommand{\NN}{\ensuremath{\mathbb N}}
\newcommand{\EXR}{\ensuremath{ \left] \minf, \pinf \right] }}
\newcommand{\CVF}{\ensuremath{\Gamma_{0}(\HH)}}
\newcommand{\BB}{\ensuremath{\mathbb B}}

\newcommand{\minf}{\ensuremath{ {-}\infty}}
\newcommand{\pinf}{\ensuremath{ {+}\infty}}
\newcommand{\RX}{\ensuremath{ \left] \minf, \pinf \right] }}

\DeclareMathOperator*{\argmin}{argmin}

\newcommand{\closu}[1]{\ensuremath{\overline{#1} }}
\newcommand{\inte}{\ensuremath{\operatorname{int}}}
\newcommand{\reli}{\ensuremath{\operatorname{ri}}}

\newcommand{\conv}{\ensuremath{\operatorname{conv}}}
\newcommand{\aff}{\ensuremath{\operatorname{aff}}}
\newcommand{\ran}{\ensuremath{\operatorname{ran}}}
\newcommand{\bdry}{\ensuremath{\operatorname{bdry}}}
\newcommand{\dom}{\ensuremath{\operatorname{dom}}}
\newcommand{\jj}{\ensuremath{\operatorname{q}}}

\newcommand{\Id}{\ensuremath{\operatorname{Id}}}
\newcommand{\prox}[1]{\ensuremath{\mathop{\operatorname{Prox}_{#1}}}}
\newcommand{\gra}{\ensuremath{\operatorname{gra}}}

\let\subset\subseteq
\newcommand{\email}[1]{\href{mailto:#1}{\nolinkurl{#1}}} % for email		
\let\oldFootnote\footnote
\newcommand\nextToken\relax
\renewcommand\footnote[1]{%
    \oldFootnote{#1}\futurelet\nextToken\isFootnote}
\newcommand\isFootnote{%
    \ifx\footnote\nextToken\textsuperscript{,}\fi}
\begin{document}

\title{\sffamily
%More on Characterizations of
On Characterizations of
(Almost) Strictly Convex Functions\footnote{Dedicated to Simeon Reich on the occasion of his 80th birthday.}
%Strictly and Almost Strictly Convex Functions
}

\author{
Heinz H. Bauschke\thanks{
Department of Mathematics, University of British Columbia, Kelowna, B.C.\ V1V~1V7, Canada.
Email: \email{heinz.bauschke@ubc.ca}.}, ~
Honglin Luo\thanks{
{School of Mathematical Sciences, Chongqing Normal University, Chongqing, PRC. Email: \email{071025013@fudan.edu.cn}.}},~
and Xianfu Wang\thanks{
Department of Mathematics, University of British Columbia, Kelowna, B.C.\ V1V~1V7, Canada.
Email: \email{shawn.wang@ubc.ca}.}
}

\date{May 4, 2026}

\maketitle

\begin{abstract}
\noindent

In this paper, we unify and improve existing results on characterizing strict and almost stricty convex functions
via subdifferential mapping, Moreau envelope, and proximal mappings.
In particular, it is shown that if a convex function is subdifferentiable on its domain, then it is strictly convex
if and only if its subdifferential is strictly monotone, equivalently, almost strictly monotone.
Rockafellar-Wets' characterizations of almost strictly convex functions via almost differentiability of
Fenchel
conjugates and strict monotonicity of subdifferentials are extended from a
finite-dimensional space to a Hilbert space.
We also establish similar results for paramonotone operators.
\end{abstract}

{\small

}

%introduction
\section{Introduction}
Throughout, $\HH$ is a real Hilbert space with inner product defined by
$\scal{x}{y}$ and induced norm $\|x\|:=\sqrt{\scal{x}{x}}$ for
$x,y\in\HH$.
%  $\EXR$ denotes $\left[-\infty,+\infty\right]$.
Strictly convex and almost strictly convex functions are important in optimization since they have
a unique global minimizer, if any. However, in the literature, results on
 strictly and almost strictly convex functions in
a general Hilbert space are sporadic; see, e.g., \cite{bb-combettes,bglw08,borwein06,Rock70,RW,volle12}.
We believe that it is very interesting to give a systematic analysis: review old results with new proofs and insights,
and provide some new results.
Let us remark that almost strictly convex functions are also known as
essentially strictly convex functions in the literature; see, e.g., \cite{bb-combettes,borwein06,Rock70}.
Our goal here is threefold. First, we consider what conditions are necessary and sufficient for characterizing
strictly convex functions and almost strictly convex functions in terms of subdifferential mappings, proximal mappings,
and Moreau envelopes. Second,
we extend the results on subdifferential mappings to paramontone operators which are maximally monotone.
Third, when characterizing an almost strictly convex function via its Fenchel conjugate, we highlight that
while in finite-dimensional spaces its conjugate has a nonempty interior domain for free, in
infinite-dimensional spaces it is a price to pay.

Let $f:\HH\rightarrow\RX$ be a proper, lsc, and convex function with subdifferential $\partial f$ and Fenchel
conjugate $f^*$. Consider the following properties:
\begin{enumerate}
\item\label{i:f:strict} $f$ is strictly convex.
\item \label{i:sub:smono} $\partial f$ is strictly monotone.
\item \label{i:sub:mono:almost} $\partial f$ is strictly monotone on convex subsets of $\dom\partial f$.
\item \label{i:f:strict:almost} $f$ is almost strictly convex.
\item \label{i:conjugate} $f^*$ is almost differentiable.
\end{enumerate}

In this paper, we show that \ref{i:f:strict}$\Rightarrow$\ref{i:sub:smono},
but the converse fails in $\RR^2$; and that
%$\Rightarrow$ \ref{i:sub:mono:almost}, but the converse fails; and that
\ref{i:sub:smono}$\Leftrightarrow$\ref{i:sub:mono:almost}$\Leftrightarrow$
\ref{i:f:strict:almost}$\Leftrightarrow$\ref{i:conjugate}.
This sheds a new light on: while strict convexity and almost strict convexity differ for convex functions (well-known),
strict monotonicity and almost strict monotonicity are the same for subdifferentials (new).
Our investigation relies on the paramonotonicity, introduced
by Censor, Isuem, and Zenios \cite{censor,iusem}.
Moreover, we show that a convex function is almost strictly
convex if and only if its subdifferential is strictly monotone in
a Hilbert space, which extends a finite-dimensional result by Rockafellar-Wets \cite{RW}.
Similar results are also established for paramonotone operators which are maximally monotone.
Almost strictly convex functions and almost differentiable functions, also known as
essentially strictly convex functions and essentially smooth
functions respectively,
have been
comprehensively studied in general Banach spaces in
% Bauschke, Borwein, and Combettes
\cite{bb-combettes}; which, however,
does not contain characterizations in terms of strict and almost strict monotonicities of subdifferential
mappings. Therefore, many results in the present paper are also new even in a Hilbert
space setting.

The paper is organized as follows. Section~\ref{s:differentiable} reviews some
facts on differentiable strictly convex functions and presents general second order derivative tests
for strict convexity. In section~\ref{s:strict:almosts} we discuss
relationship between strictly and almost strictly convex functions. Section~\ref{s:strict:convex}
provides an account of the main properties of strictly convex functions via subgradient inequalities and
subdifferential mappings.
In Section~\ref{s:almostfunction} we characterize almost strictly convex functions via the
strict monotonicity of
their subdifferential
mappings, which generalize the result by Rockafellar-Wets from a finite-dimensional space to a general Hilbert space.
Moreover, we show that the subdifferential mapping of a proper, lsc, and convex function is almost strictly
monotone if and only if it is strictly monotone.
 In Section~\ref{s:paramonotone}, the general results of section~\ref{s:almostfunction}
are extended to paramontone operators. Notably, an almost strict convexity versus almost
differentiability duality is given in
Section~\ref{s:duality}. Moreover, different variants of almost differentiable functions are shown to be the same.
Section~\ref{s:openprobs} is devoted to the connection between almost strictly convex functions and tilted-stable optimization, and open problems.

In the remainder of this section we recall some basic concepts used in the sequel. Our notations
basically follow \cite{Bauschke-Combettes-2017, RW}.

\subsection{Convex functions and monotone operators}
\begin{definition} A function $f:\HH\rightarrow\EXR$ is convex if $\dom f$ is convex
and
$$(\forall x, y\in\dom f)
(\forall \lambda\in \left]0,1\right[)\ f(\lambda x+(1-\lambda)y)\leq
\lambda f(x)+(1-\lambda)f(y);$$
and $f$ is strictly convex if $\dom f$ is convex
and
$$(\forall x, y\in\dom f, x\neq y)(\forall \lambda\in \left]0,1\right[)\ f(\lambda x+(1-\lambda)y)<\lambda f(x)+(1-\lambda)f(y).$$
\end{definition}
We call $f$ a proper function if $f(x)<+\infty$ for at least one $x\in\HH$, and $f(x)>-\infty$ for all $x\in\HH$.
The set of proper, lsc, and convex functions from $\HH$ to $\RX$ is denoted by $\Gamma_{0}(\HH)$. For $f\in\CVF$,
its subdifferential mapping is the set-valued operator
$$\partial f:\HH\rightarrow 2^{\HH}:
x\mapsto \menge{v\in\HH}{(\forall y\in\HH)\ f(y)\geq f(x)+\scal{v}{y-x}},$$
and its directional derivative at $x\in\dom f$ in the direction $u\in \HH$ is
$$f'(x;u)=\lim_{t\downarrow 0}\frac{f(x+tu)-f(x)}{t}.$$
Moreover, $v\in\partial f(x)\Leftrightarrow [(\forall u\in\HH)\ f'(x;u)\geq \scal{v}{u}].$
The Fenchel conjugate of $f$ is given by
$$f^*:\HH\rightarrow\RX:v\mapsto \sup_{x\in\HH}(\scal{x}{v}-f(x)),$$
$f^*\in\CVF$, and $\partial f^*=(\partial f)^{-1}$; see \cite[Corollaries 13.38, 16.29]{Bauschke-Combettes-2017}.
Subdifferential mapping and Fenchel conjugate play an important role in convex analysis and optimization; see, e.g.,
\cite{Bauschke-Combettes-2017, BeckSIAM, hiriart93i, hiriart93ii, Rock70, RW}.
It is well-known that for $f\in\CVF$, $\dom\partial f$ might not be convex; see, e.g.,
\cite{moffat16, Rock70}.
\begin{definition} A function $f\in\CVF$ is almost strictly convex if it is strictly convex
along every line segment in $\dom\partial f$.
\end{definition}
Different from \cite[Definition 5.2(ii)]{bb-combettes}, we do not require $(\partial f)^{-1}$ to be locally bounded,
in particular, $\inte\dom(\partial f)^{-1}=\inte\dom f^*\neq\varnothing$, which is automatic for an almost strictly
convex function $f$ on a finite-dimensional space, see Lemma~\ref{l:finite} in section~\ref{s:duality}.
\begin{example}\emph{(\cite[Example 5.14]{bb-combettes})}\label{e:empty:interior}
 Let $(p_{n})_{n\in\NN}$ be a sequence in $\left[2,+\infty\right[$
and $p_{n}\rightarrow\infty$
as $n\rightarrow\infty$. In the Hilbert space $\ell^{2}$, define
$$f:\ell^2\rightarrow\RR: x\mapsto \sum_{n}\frac{1}{p_{n}}|x_{n}|^{p_{n}}.$$
Then $f$ is strictly convex and G\^ateaux differentiable. The Fenchel conjugate of $f$ is
$$f^*:\ell^2\rightarrow\RX:y\mapsto \sum_{n}\frac{1}{q_{n}}|y_{n}|^{q_{n}}$$
where $1/p_{n}+1/q_{n}=1$. Then $f^*$ is strictly convex, and
$\partial f^*$ is at most single-valued and $\inte\dom f^*=\varnothing$.
\end{example}

\begin{definition}
An operator $A:\HH\rightarrow 2^{\HH}$ is monotone if
$$(\forall (x_{i},v_{i})\in\gra A, i=0,1)\ \scal{x_{0}-x_{1}}{v_{0}-v_{1}}\geq 0;$$
strictly monotone if
$$[(\forall (x_{i},v_{i})\in\gra A, i=0,1)\ x_{0}\neq x_{1}]\ \Rightarrow\ \scal{x_{0}-x_{1}}{v_{0}-v_{1}}>0.$$
% almost strictly monotone if $A$ is strictly monotone on
%every line segment in $\dom A$.
\end{definition}
\begin{definition}
An operator $A:\HH\rightarrow 2^{\HH}$ is maximally monotone if no enlargement of its graph
is possible in $\HH\times\HH$ without destroying monotonicity, i.e., for every pair
$$(\forall (x,v)\in\HH\times\HH\setminus\gra A)(\exists (\bar x,\bar v)\in \gra A)\ \scal{\bar v- v}{\bar x-x}<0.$$
\end{definition}
A fundamental example of a maximally monotone operator is the subdifferential
of a function in $\CVF$; see, e.g., \cite{Bauschke-Combettes-2017,Rock70,simons}.
It is well-known that even the domain of a maximally monotone operator $A:\HH\rightarrow 2^{\HH}$ might
not be convex; see, e.g., \cite{moffat16,Rock70,simons}. Therefore,
it is natural to introduce the following notion.
\begin{definition} A monotone operator $A:\HH\rightarrow 2^{\HH}$ is almost strictly monotone
if it is strictly monotone on
every line segment in $\dom A$.
\end{definition}

\subsection{Notations}
In Hilbert space $\HH$,
the expressions $x_{k}\rightharpoonup x$ and $x_{k}\rightarrow x$ denote, respectively, the weak
and strong convergence to $x$ of a sequence $(x_{k})_{k\in\NN}$ in $\HH$.
 For any two different points $x_{0}, x_{1}\in\HH$, the closed line segment is
$[x_{0},x_{1}]:=\menge{(1-t)x_{0}+tx_{1}}{0\leq t\leq 1}$, and the open line segment
is $\left]x_{0},x_{1}\right[:=\menge{(1-t)x_{0}+tx_{1}}{0<t<1}$.
For an operator $A:\HH\rightarrow 2^{\HH}$, its graph is
$\gra A:=\menge{(x,v)\in\HH\times\HH}{v\in Ax}$, domain $\dom A:=\menge{x\in\HH}{Ax\neq\varnothing}$,
and range $\ran A:=\menge{v\in\HH}{v\in Ax, x\in\dom A}$.
The set-valued inverse of $A$ is $A^{-1}$ with $\gra A^{-1}:=\menge{(v,x)\in\HH\times\HH}{(x,v)\in\gra A}.$
The resolvent of $A$ is denoted by $J_{A}:=(\Id+A)^{-1}$, where $\Id$ is the identity operator on $\HH$.
When $A:\HH\rightarrow\HH$ is a linear operator, $A^*$ is its adjoint.
For a function $f\in\CVF$, we use
$\dom f, \dom \partial f, \ran\partial f$ for its domain, subdifferential domain,
and subdifferential range.
Let $C$ be a subset of $\HH$. Its interior is denoted by $\inte C$, closure in the strong topology by
$\closu{C}$, closure of its convex hull in the strong topology by $\closu{\conv}C$, and closure of its convex hull in the weak topology by $\closu{\conv}^w C$.
When $\HH$ is finite-dimensional, the relative interior of $C$ is
$\reli C:=\menge{x\in\HH}{(\exists \delta>0)\ \BB_{\delta}(x)\cap \aff(C)\subset C}$
in which $\BB_{\delta}(x)$ is the ball centered at $x$ with radius $\delta>0$ and $\aff C$ is the affine hull of $C$.
Finally, the normal cone mapping of $C$ is defined by
$N_{C}(x):=\menge{v}{(\forall y\in C)\ \scal{v}{y-x}\leq 0}$ if $x\in C$, and $\varnothing$ if $x\not\in C$.

\section{Differentiable and strictly convex functions on $\RR^n$}\label{s:differentiable}
We start with some facts and results on strictly convex functions on $\RR$.  These facts can be found
in \cite[pages 45--47]{RW}. More general second order derivative tests for strict convexity are also given.
\begin{fact}\emph{(\cite[Theorem 2.13]{RW})}\label{f:one:dim}
For a differentiable function $f$ on an open interval $I\subseteq\RR$, each of the following is sufficient and necessary
for $f$ to be strictly convex on $I$:
\begin{enumerate}
\item\label{i:oned1} $f'$ is strictly increasing on $I$.
\item $f(y)>f(x)+f'(x)(y-x)$ for all $x,y\in I$ with $x\neq y$.
\end{enumerate}
A sufficient  (but not necessary) condition is:
 $f"(x)>0$ for all $x\in I$ (assuming twice differentiability).
\end{fact}
Below is a more general second order derivative test for strict convexity.
\begin{theorem}\label{t:main}
 Let $f:[a,b]\rightarrow\RR$ be piecewise differentiable. Suppose that $f'$ is differentiable
almost everywhere.
If $f"\geq 0$ almost everywhere, and the Lebesgue measure
$\mu(I\cap\{x|\ f"(x)>0\})>0$ for very open interval $I$, then
$f$ is strictly convex.
\end{theorem}
\begin{proof} We only need to show that $f'$ is strictly increasing by Fact~\ref{f:one:dim}\ref{i:oned1}.
Indeed, for $x<y$,
by the assumption, we have
$$f'(y)-f'(x)=\int_{x}^y f"(t)dt>0.$$
\end{proof}
\begin{corollary} Let $f:[a,b]\rightarrow\RR$ be piecewise differentiable. Suppose that $f'$ is differentiable
except for a set of countable points.
If $f">0$ whenever it exists, then
$f$ is strictly convex.
\end{corollary}

\begin{corollary} Let $f:[a,b]\rightarrow\RR$ be piecewise differentiable. Suppose that $f'$ is differentiable
except for a set of finite points.
If $f">0$ whenever it exists, then
$f$ is strictly convex.
\end{corollary}

The converse of Theorem~\ref{t:main} fails, as the following example shows.
\begin{example}
Let $h:\RR\rightarrow\RR$ be a Cantor-type singular function, i.e., continuous, strictly increasing,
and $h'=0$ almost everywhere \cite{stromberg}. Define
$f(x):=\int_{0}^x h(t)dt.$
Then $f$ is strictly convex, but $f"=0$ almost everywhere.
\end{example}

Next we consider strictly convex functions on $\RR^n$.

\begin{fact}\emph{(\cite[Theorem 2.14]{RW}, \cite[Theorem 2.1.12]{zalinescu})}\label{f:strict:convex}
For a differentiable function $f$ on an open convex set $O\subseteq\RR^n$, each of the following is sufficient and necessary
for $f$ to be strictly convex on $O$:
\begin{enumerate}
\item\label{i:twod1}
 $\scal{\nabla f(x_{1})-\nabla f(x_{2})}{x_{1}-x_{2}}>0$ for all $x_{1},x_{2}\in O$ with $x_{1}\neq x_{2}$.
\item\label{i:tangentline}  $f(x_{2})>f(x_{1})+\scal{\nabla f(x_{1})}{x_{2}-x_{1}}$ for all $x_{1},x_{2}\in O$ with $x_{1}\neq x_{2}$.
\end{enumerate}
A sufficient  (but not necessary) condition is:
its Hessian matrix $\nabla^2 f(x)$ is positive definite for all $x\in O$ (assuming twice differentiability).
\end{fact}

The following extends Theorem~\ref{t:main} to higher dimensional spaces.
\begin{theorem}\label{t:main1} Let $O\subseteq\RR^n$ be a nonempty open convex set, and let
 $f:O\rightarrow\RR$ be differentiable. Suppose that $\nabla f$ is differentiable
almost everywhere.
If $\nabla^2 f(x)$ is positive semidefinite for almost everywhere $x\in O$,
and the Lebesgue measure
\begin{equation}\label{e:pmeasure}
\mu(\{t|\ t\in [0,1], \nabla^2 f(x_{1}+t(x_{2}-x_{1})) \text{ is positive definite}\})>0
\end{equation}
for all $x_{1},x_{2}\in O$, then
$f$ is strictly convex.
\end{theorem}
\begin{proof} We only need to show that $\nabla f$ is strictly monotone by Fact~\ref{f:strict:convex}\ref{i:twod1}.
Indeed, for $x_{1},x_{2}\in O$ with
$x_{1}\neq x_{2}$,
by the assumption, we have
$$\scal{\nabla f(x_{2})-\nabla f(x_{1})}{x_{2}-x_{1}}=\int_{0}^1\scal{\nabla^2f(x_{1}+t(x_{2}-x_{1}))(x_{2}-x_{1})}{x_{2}-x_{1}}dt>0.$$
\end{proof}

\section{Strictly versus almost strictly convex functions}\label{s:strict:almosts}
This section is devoted to relationship between strictly and almost strictly convex functions.
Clearly, a strictly convex function is almost strictly convex.
Observe that $\partial f(x)\neq\varnothing$ for $x\in\reli\dom f$ in $\RR^n$ ($\inte\dom f$ in $\HH$) and that
$\reli\dom f$ (
$\inte\dom f$) is convex. If $f$ is almost strictly convex it must be strictly convex on
$\reli\dom f$ in $\RR^n$ ($\inte\dom f$ in $\HH$). Thus, an almost strictly convex function
$f$ can  fail to be strictly convex only on line segments in the relative boundary or
boundary of $\dom f$.

On the real line,
the following pleasing result holds.
\begin{theorem}\label{t:line:same}
Let $f\in\Gamma_{0}(\RR)$. Then the following are equivalent:
\begin{enumerate}
\item\label{i:one1}
$f$ is almost strictly convex.
\item\label{i:one2} $f$ is strictly convex on $\inte\dom f$.
\item\label{i:one3} $f$ is strictly convex.
\end{enumerate}
\end{theorem}
\begin{proof} ``\ref{i:one1}$\Rightarrow$\ref{i:one2}":
  Suppose that $f$ is almost strictly convex.
Observe that $\dom f$ is an interval.
 Because $f$ is subdifferentiable on $\inte\dom f$ and
$\inte\dom f$ is an interval, $f$ is strictly convex on $\inte\dom f$.

``\ref{i:one2}$\Rightarrow$\ref{i:one3}":
%This is equivalent to that $f$ is strictly convex on $\dom f$.
To see this, consider three cases:

\noindent\emph{Case 1: $\dom f$ is an open interval.} This is clear.

\noindent\emph{Case 2: $\dom f=\left[a,b\right]$.}
Suppose that $f$ is affine on $\left[a, \gamma\right]$ with $a<\gamma\leq b$. Then
$f$ is affine on $\left]\alpha,\gamma\right[$, which is a contradiction. Suppose that $f$ is affine on $\left[\gamma, b\right]$ with $a\leq \gamma< b$. Then
$f$ is affine on $\left]\gamma, b\right[$, which is a contradiction.
%Suppose $f$ is affine on $[a, b]$. Then
%$f$ is affine on $]a, b[$, which is a contradiction.

\noindent\emph{Case 3: $\dom f=\left[a,b\right[$ allowing $b=+\infty$,
and $\left]a, b\right]$ allowing $a=-\infty$.} The arguments are similar
as in \emph{Case 2}.
%Using left and right continuity of $f$ at the endpoints, we deduce that $f$ is strictly convex on $\dom f$.

``\ref{i:one3}$\Rightarrow$\ref{i:one1}": This is clear.
\end{proof}

A remarkable example due to Rockafellar in $\RR^2$  is in order. Denote
the positive orthant in $\RR^n$ by $\RR^n_{++}:=\menge{(x_{i})_{i=1}^{n}}{x_{i}>0, i=1,\ldots, n}$.
%comes first:
\begin{example}
Define $f:\RR^2\rightarrow\RX$ by
$$(x_{1},x_{2})\mapsto \begin{cases}
x_{2}^2/(2x_{1})-2x_{2}^{1/2}, &\text{ if $x_{1}>0$, $x_{2}\geq 0$;}\\
0, &\text{ if $x_{1}=x_{2}=0$;}\\
+\infty, &\text{ otherwise.}
\end{cases}
$$
Then $f$ is strictly convex $\dom \partial f=\RR^2_{++}$, but $f(x_{1},0)=0$ for $x_{1}\geq 0$ not strictly convex.
Hence $f$ is almost strictly convex but not strictly convex. See \cite[page 253]{Rock70} for further details.
\end{example}
Although strictly and almost strictly convex functions are different, they share a surprising common property: unique minimizer, if any. This is one of the key motivation for optimizers to study them.
\begin{theorem} Suppose that $f\in\CVF$ verifies one of the following conditions:
\begin{enumerate}
\item\label{i:min1} $f$ is strictly convex.
\item\label{i:min2} $f$ is almost strictly convex.
\end{enumerate}
Then $(\forall x^*\in\HH)$ the function $x\mapsto f(x)-\scal{x^*}{x}$ has at most one minimizer.
\end{theorem}
\begin{proof}
\ref{i:min1}: The function $x\mapsto f(x)-\scal{x^*}{x}$ is strictly convex.

\ref{i:min2}: Suppose that $\argmin_{x\in\HH}\big(f(x)-\scal{x^*}{x}\big)$ contains two point
$x_{0},x_{1}$ with $x_{0}\neq x_{1}$. Then $x_{0},x_{1}\in\partial f^*(x^*)$ so that
$[x_{0},x_{1}]\subset \partial f^*(x^*)$. Since $x^*\in\partial f(x)$ for every $x\in [x_{0},x_{1}]$ we deduce that
$[x_{0}, x_{1}]\subset\dom\partial f$ and
$f(x)-\scal{x^*}{x}$ attains global min on $[x_{0},x_{1}]$. Thus,
$$(\exists c\in\RR)(\forall x\in [x_{0},x_{1}])\ f(x)=\scal{x^*}{x}+c,$$
contradicting that $f$ is almost strictly convex.
\end{proof}

\section{Strictly convex functions on $\HH$}\label{s:strict:convex}

In general, characterizing a strictly convex function via its subdifferential mapping is murky,
unless the domain of its subdifferential is nice.
We begin with a strict gradient inequality result for a strictly convex function.

\begin{lemma}\label{l:inequality:s}
Suppose that $f\in \CVF$ is strictly convex and $x_{0}\in\dom\partial f$. Then
$$(\forall x\in\HH\setminus\{x_{0}\})(\forall v\in\partial f(x_{0})) \
f(x)>f(x_{0})+\scal{v}{x-x_{0}}.$$
\end{lemma}
\begin{proof}
Let $x\neq x_{0}$ and
$v\in\partial f(x_{0})$. The function
$$\varphi:[0,1]\rightarrow\RR: t\mapsto f((1-t)x_{0}+tx)$$ is strictly convex. Using
\cite[Theorem 2.1.13]{zalinescu} and $(\forall u\in\HH)\ f'(x_{0};u)\geq \scal{u}{v}$, we
obtain
$f(x)-f(x_{0})=\varphi(1)-\varphi(0)>\varphi'_{+}(0)
=f'(x_{0};x-x_{0})\geq \scal{v}{x-x_{0}}.$
\end{proof}

This gives the following well-known but one-way fact, see, e.g.,  \cite[Example 22.4(ii)]{Bauschke-Combettes-2017}.

\begin{theorem}\label{l:strictly:mono} Suppose that $f\in \CVF$ is strictly convex. Then
$\partial f$ is strictly monotone.
\end{theorem}
\begin{proof} Let $x_{0}\neq x_{1}$ and $v_{i}\in\partial f(x_{i})$ with $i=0,1$. Applying Lemma~\ref{l:inequality:s}
gives
\begin{equation}\label{e:strict:1}
f(x_{1})>f(x_{0})+\scal{v_{0}}{x_{1}-x_{0}}, \text{ and }
\end{equation}
\begin{equation}\label{e:strict:2}
f(x_{0})>f(x_{1})+\scal{v_{1}}{x_{0}-x_{1}}.
\end{equation}
The conclusion follows by adding \eqref{e:strict:1} and \eqref{e:strict:2}.
%We prove by contradiction. Suppose $\partial f$ is not strictly monotone, i.e.,
% $\exists v_{0}\in\partial f(x_{0}), v_{1}\in\partial f(x_{1})$, $x_{0}\neq x_{1}$ such that
%$\scal{v_{0}-v_{1}}{x_{0}-x_{1}}=0.$
%Apply Lemma~\ref{l:not:strict}\ref{i:faffine}.
\end{proof}

The following result on a general convex function is of independent interest.
\begin{lemma}\label{l:affine:segment} Let $f\in \CVF$.
 Suppose that $f$ is affine on $[x_0,x_{1}]$. Then
$$(\forall x\in \left]x_{0},x_{1}\right[)\ \partial f(x)=\partial f(x_{0})\cap \partial f(x_{1}).$$
\end{lemma}
\begin{proof}
Fix $\barx\in \left]x_{0},x_{1}\right[$.
We first show $\partial f(\barx)\subset\partial f(x_{0})\cap \partial f(x_{1})$.
Let $v\in\partial f(\barx)$ (if nonempty).
Define $g:\RR\rightarrow\EXR$ by
$t\mapsto f(\barx+t(x_{1}-x_{0}))-f(\barx)$. Then $g$ is linear on the interval $[\alpha,\beta]$ with
$\barx+\alpha (x_{1}-x_{0})=x_{0}$ and $\barx+\beta (x_{1}-x_{0})=x_{1}$;
in particular, $0\in \left]\alpha,\beta\right[$ because
of $\barx\in \left]x_{0},x_{1}\right[$. By the convexity of $f$,
$g(t)\geq \scal{v}{\bar{x}+t(x_{1}-x_{0})-\barx}=t\scal{v}{x_{1}-x_{0}}$ on $[\alpha,\beta]$.
Since $0\in \left]\alpha,\beta\right[$,
we deduce
$g(t)=t\scal{v}{x_{1}-x_{0}}$ on $[\alpha,\beta]$. It follows that
\begin{equation*}
(\forall t\in[\alpha,\beta])\ f(\barx+t(x_{1}-x_{0}))-f(\barx)=t\scal{v}{x_{1}-x_{0}}=\scal{v}{\barx+t(x_{1}-x_{0})-\barx},
\end{equation*}
i.e., $(\forall y\in [x_{0},x_{1}])\ f(y)-f(\barx)=\scal{v}{y-\barx}.$
In particular,
\begin{align*}
f(x_{1})-f(\barx) &=\scal{v}{x_{1}-\barx}, \text{ and }\\
f(x_{0})-f(\barx) &=\scal{v}{x_{0}-\barx}.
\end{align*}
Now
\begin{align*}
(\forall y\in\RR^n)\
f(y)& \geq f(\barx)+\scal{v}{y-\barx}=f(\barx)+\scal{v}{y-x_{0}}+\scal{v}{x_{0}-\barx}\\
&=f(x_{0})+\scal{v}{y-x_{0}},
\end{align*}
so $v\in\partial f(x_{0}).$ Similarly,
\begin{align*}
(\forall y\in\RR^n)\
f(y)& \geq f(\barx)+\scal{v}{y-\barx}=f(\barx)+\scal{v}{y-x_{1}}+\scal{v}{x_{1}-\barx}\\
&=f(x_{1})+\scal{v}{y-x_{1}},
\end{align*}
so $v\in\partial f(x_{1}).$ Hence $v\in\partial f(x_{0})\cap \partial f(x_{1})$.

Conversely, we show $\partial f(x_{0})\cap \partial f(x_{1})\subset \partial f(\barx)$.
Let $v\in\partial f(x_{0})\cap \partial f(x_{1})$ (if nonempty). This gives
$x_{0}, x_{1}\in \partial f^*(v)$ so that $[x_{0}, x_{1}]\subset\partial f^*(v)$, in particular,
$\barx \in \partial f^*(v)$ , i.e., $v\in\partial f(\barx)$.
%Since $\partial f$ is paramonotone, $x_{0}, x_{1},\barx\in\dom \partial f$,
%$\barx\in ]x_{0},x_{1}[$ and $v\in\partial f(x_{0})\cap \partial f(x_{1})$,
%by \cite[Proposition 22.10(ii)]{Bauschke-Combettes-2017}, we have
%$\partial f(\barx)=\partial f(x_{0})\cap \partial f(x_{1})$,
%so $v\in\partial f(\barx)$.
\end{proof}

This leads to the following amazing result on $\RR$.

\begin{corollary}
Consider $f:\RR\rightarrow \EXR$. Then $f$ is strictly convex if and only if
$\partial f$ is strictly monotone.
\end{corollary}
\begin{proof}
``$\Rightarrow$": Suppose that $f$ is strictly convex. Apply Theorem~\ref{l:strictly:mono}.
``$\Leftarrow$": Suppose that $\partial f$ is strictly monotone. Observe that $f$ is subdifferentiable on
the open interval $\inte\dom f$. The assumption implies that $f$ is strictly convex on
$\inte\dom f$ by Lemma~\ref{l:affine:segment}. It remains to apply Theorem~\ref{t:line:same}.
\end{proof}

The following example by Rockafellar \cite[page 253]{Rock70}
shows that the converse of Theorem~\ref{l:strictly:mono}
fails on $\RR^n$ when $n\geq 2$.
\begin{example}\label{e:almost:strict}
Define $f:\RR^2\rightarrow\EXR$ by
$$(x_{1},x_{2})\mapsto \begin{cases}
x_{2}^2/(2x_{1})-2x_{2}^{1/2}, &\text{ if $x_{1}>0$, $x_{2}\geq 0$;}\\
0, &\text{ if $x_{1}=x_{2}=0$;}\\
+\infty, &\text{ otherwise.}
\end{cases}
$$
Then
\begin{enumerate}
\item $\dom \partial f=\dom\nabla f=\RR^{2}_{++}$.
\item $\partial f$ is strictly monotone.
\item While $f$ is strictly convex on $\dom \partial f=\RR^2_{++}$,
 $f(x_{1},0)=0$ for $x_{1}\geq 0$ is not strictly convex.
Hence $f$ is not strictly convex.
\end{enumerate}
\end{example}
\begin{remark} \
\begin{enumerate}
\item Example~\ref{e:almost:strict} also shows that for $f\in \CVF$, the condition
$$(\forall y\in\HH\setminus\{x\})(\forall v\in\partial f(x))\ f(y)>f(x)+\scal{v}{y-x}$$
does not necessarily imply that $f$ is strictly convex. That is,
Fact~\ref{f:strict:convex}\ref{i:tangentline} fails for a general convex function.
The reason is that $\dom \partial f$ might not
be equal to $\dom f$.
\item Example~\ref{e:almost:strict} also shows that Theorem~\ref{t:main1} fails
without condition~\eqref{e:pmeasure}. That is,  $\nabla^2 f$ being positive definite
almost everywhere does not imply that $f$ is strictly convex.
\end{enumerate}
\end{remark}

For a finite-valued convex function on an open convex set, the
following pleasant result,
which generalizes Fact~\ref{f:strict:convex},  holds.
\begin{corollary}\label{c:nice}
Let $f\in\CVF$ and let $\dom f$ contain an open convex subset $O$ of $\HH$.
Then the following are equivalent:
\begin{enumerate}
\item\label{i:newyear1} $f$ is strictly convex on $O$.
\item\label{i:newyear2} $\partial f$ is strictly monotone on $O$.
\item\label{i:newyear3}  $f(x_{1})>f(x_{0})+\scal{v}{x_{1}-x_{0}}$ for all $x_{1},x_{0}\in O$ with
$x_{1}\neq x_{0}$ and
$v\in\partial f(x_{0})$.
\end{enumerate}
\end{corollary}

\begin{proof} According to \cite[Corollary 8.39]{Bauschke-Combettes-2017}, $f$ is continuous
on the open convex set $O$. This ensures that $(\forall x\in O)\ \partial f(x)\neq\varnothing$
by \cite[Proposition 16.17(ii)]{Bauschke-Combettes-2017}.

``\ref{i:newyear1}$\Rightarrow$\ref{i:newyear3}'': Let $x_{0}\neq x_{1}$ and
$v\in\partial f(x_{0})$. The function
$$\varphi:[0,1]\rightarrow\RR: t\mapsto f((1-t)x_{0}+tx_{1})$$ is strictly convex. Using
\cite[Theorem 2.1.13]{zalinescu} and $(\forall u\in\HH)\ f'(x_{0},u)\geq \scal{u}{v}$, we
obtain
$\varphi(1)-\varphi(0)>\varphi'_{+}(0)
=f'(x_{0},x_{1}-x_{0})\geq \scal{v}{x_{1}-x_{0}}.$

``\ref{i:newyear3}$\Rightarrow$\ref{i:newyear1}":
Let $x_{0}\neq x_{1}$ and $t\in \left]0,1\right[$. Take $v\in\partial f((1-t)x_{0}+tx_{1})$, which is possible because $(1-t)x_{0}+tx_{1}\in O$.
We have
\begin{equation}\label{e:strict1}
f(x_{0})> f((1-t)x_{0}+tx_{1})+\scal{v}{x_{0}-((1-t)x_{0}+tx_{1})}=f((1-t)x_{0}+tx_{1})+t\scal{v}{x_{0}-x_{1})},
\end{equation}
\begin{equation}\label{e:strict2}
f(x_{1})> f((1-t)x_{0}+tx_{1})+\scal{v}{x_{1}-((1-t)x_{0}+tx_{1})}=f((1-t)x_{0}+tx_{1})+(1-t)\scal{v}{x_{1}-x_{0})}.
\end{equation}
Multiplying \eqref{e:strict1} by $(1-t)$ and \eqref{e:strict2} by $t$, followed by adding them, we obtain
$$(1-t)f(x_{0})+t f(x_{1})> f((1-t)x_{0}+tx_{1}),$$
as required.

%Combine Lemma~\ref{l:domandsub} and Corollary~\ref{c:strict:cm}.

``\ref{i:newyear2}$\Rightarrow$\ref{i:newyear3}": By the
Mean Value Theorem \cite[Theorem 16.56]{Bauschke-Combettes-2017},
$\exists t\in \left]0,1\right[$ and $u\in\partial f(x_{0}+t(x_{1}-x_{0}))$
such that $f(x_{1})-f(x_{0})=\scal{u}{x_{1}-x_{0}}$.
Then for every $v\in\partial f(x_{0})$,
\begin{align*}
& f(x_{1})-f(x_{0})-\scal{v}{x_{1}-x_{0}}\\
&=\scal{u-v}{x_{1}-x_{0}}=t^{-1}\scal{u-v}{(x_{0}+t(x_{1}-x_{0}))-x_{0}}>0
\end{align*}
by the assumption.

``\ref{i:newyear3}$\Rightarrow$\ref{i:newyear2}": This is clear.
\end{proof}

\section{Almost strictly convex functions on $\HH$}\label{s:almostfunction}
This section concerns characterizations of almost strictly convex functions
via subdifferential mappings.
We begin with the following simple result, which characterize an almost strict convex function via
its subgradient inequalities.
\begin{theorem}\label{t:inequality:s} Let $f\in\CVF$. Then the following are equivalent:
\begin{enumerate}
\item\label{i:almostsf} $f$ is almost strictly convex.
\item\label{i:sgradient}
 For all $[x_{0},x_{1}]\subset\dom\partial f$, $x_{0}\neq x_{1}$, $v\in\partial f(x_{0})$,
we have
$f(x_{1})>f(x_{0})+\scal{v}{x_{1}-x_{0}}.$
\end{enumerate}
\end{theorem}
\begin{proof}
The proof is similar to
that of Corollary~\ref{c:nice}\ref{i:newyear1}$\Leftrightarrow$\ref{i:newyear3}.
%$\ref{i:sgradient}\Rightarrow\ref{i:sgradient:all}$:
\end{proof}

Characterization of an almost strictly convex function via strict monotonicity of
its subdifferential mapping is beautifully clean,
as the following fact due to Rockafellar and Wets shows.

\begin{fact}\emph{(\cite[Theorem 12.17]{RW})}\label{f:rock-wets}
For any proper, convex function
$f:\RR^n\rightarrow\EXR$, the mapping $\partial f:\RR^n \to 2^{\RR^n}$ is monotone.
Indeed, a proper, lsc function $f$ is convex if and only if $\partial f$ is monotone,
in which case $\partial f$ is maximal monotone.
Such a function $f$ is almost strictly
convex if and only if $\partial f$ is strictly monotone.
\end{fact}
%We will show that ``A proper, lsc, and convex function $f:\RR^n\rightarrow\EXR$
%is almost strictly convex if and only if $\partial f$ is
%strictly monotone" is wrong!

%A correct statement comes as follows.
Below is our main result in this section, which gives
 some new characterizations of almost strictly convex functions
 via subdifferential mappings
in a general Hilbert space.
Moreover, the result extends Fact~\ref{f:rock-wets} to a general Hilbert space.
\begin{theorem}\label{t:almost} Let $f\in \CVF$. Then
the following are equivalent:
\begin{enumerate}
\item\label{i:falmost} $f$
is almost strictly convex.
\item\label{i:almost} $\partial f$ is
strictly monotone on any convex subset of $\dom\partial f$.
\item\label{i:full} $\partial f$ is strictly monotone.
\end{enumerate}
\end{theorem}

To show this, we need a few auxiliary results.

%\begin{theorem} A proper, lsc, and convex function $f:\RR^n\rightarrow\RX$
%is strictly convex if and only if $\partial f$ is
%strictly monotone.
%\end{theorem}

\begin{lemma}\label{l:not:strict}
Let $f\in \CVF$ and let
$\scal{v_{0}-v_{1}}{x_{0}-x_{1}}=0$ for $(x_{i},v_{i})\in\gra \partial f$ with $i=0,1$.
Then the following hold:
\begin{enumerate}
\item\label{i:faffine} $f$ is affine on $[x_{0},x_{1}]$.
\item\label{i:insubdomain} $[v_{0},v_{1}]\subset\partial f(x)$ for every $x\in [x_{0},x_{1}]$.
In particular, $[x_{0},x_{1}]\subset\dom\partial f$.
\item\label{i:inside} For every $x\in ]x_{0},x_{1}[$, $\partial f(x)=\partial f(x_{0})\cap\partial f(x_{1})\supset [v_{0},v_{1}]$.
\end{enumerate}
\end{lemma}

\begin{proof}
\ref{i:faffine}:
By the convexity of $f$, we have
\begin{align}
f(x_{1}) &\geq f(x_{0})+\scal{v_{0}}{x_{1}-x_{0}}\label{e:one}\\
f(x_{0}) &\geq f(x_{1})+\scal{v_{1}}{x_{0}-x_{1}}.
\end{align}
Then
\begin{align}
f(x_{0}) &\geq f(x_{1})+\scal{v_{1}}{x_{0}-x_{1}}=f(x_{1})+\scal{v_{1}-v_{0}}{x_{0}-x_{1}}+\scal{v_{0}}{x_{0}-x_{1}}\\
& =f(x_{1})+\scal{v_{0}}{x_{0}-x_{1}}. \label{e:two}
\end{align}
Combining \eqref{e:one} and \eqref{e:two} gives
\begin{equation}\label{e:affine:agree}
f(x_{1})=f(x_{0})+\scal{v_{0}}{x_{1}-x_{0}}
\end{equation}
Now for $t\in [0,1]$,
\begin{align*}
f((1-t)x_{0}+tx_{1}) & \geq f(x_{0})+\scal{v_{0}}{t(x_{1}-x_{0})}\\
&=f(x_{0})+t(f(x_{1})-f(x_{0}))=(1-t)f(x_{0})+tf(x_{1}).
\end{align*}
Since the converse always holds, we obtain
$$f((1-t)x_{0}+tx_{1})=(1-t)f(x_{0})+tf(x_{1}),$$
i.e., $f$ is affine on $[x_{0},x_{1}]$.
%This contradicts that $f$ is strictly convex.

\ref{i:insubdomain}: This follows from the paramonotonicity
of $\partial f$, \cite[Example 22.4 and Proposition 22.10(ii)]{Bauschke-Combettes-2017}.
Here we give a direct proof.
By \eqref{e:affine:agree},
$$f(x_{1})-\scal{v_{0}}{x_{1}}=f(x_{0})-\scal{v_{0}}{x_{0}}$$
Since $v_{0}\in\partial f(x_{0})$, we have $f^*(v_{0})=\scal{v_0}{x_0}-f(x_{0})$, so that
$f(x_{1})-\scal{v_{0}}{x_{1}}=-f^*(v_{0})$, i.e.,
$f(x_{1})+f^*(v_{0})=\scal{v_{0}}{x_{1}}$. Thus, $v_{0}\in\partial f(x_{1})$. Similarly,
$v_{1}\in\partial f(x_{0})$.
Define
$g(x):=f(x)-\scal{v_{0}}{x}$. Then $0\in \partial g(x_{0})\cap \partial g(x_{1})$, i.e.,
$x_{0}, x_{1}\in\argmin g$. Since $\argmin g$ is a convex set, we deduce that $[x_{0},x_{1}]\in\argmin g$,
implying $v_{0}\in\partial f(x)$ for every $x\in [x_{0},x_{1}]$. Similarly,
$v_{1}\in\partial f(x)$ for every $x\in [x_{0},x_{1}]$. Because $\partial f$ is maximally monotone,
$\partial f(x)$ is closed convex, so
$[v_{0},v_{1}]\subset\partial f(x)$.

\ref{i:inside}: By \cite[Example 22.4]{Bauschke-Combettes-2017}, $\partial f$ is paramonotone.
Since $\partial f(x_{0})\cap\partial f(x_{1})\supset\{v_{0}, v_{1}\}$ and
$[x_{0},x_{1}]\subset\dom\partial f$ by
\ref{i:insubdomain} and $x\in \left]x_{0},x_{1}\right[$,
using \cite[Proposition~22.10(ii)]{Bauschke-Combettes-2017}
we obtain $\partial f(x)=\partial f(x_{0})\cap\partial f(x_{1})$. The remaining result follows from
\ref{i:insubdomain}.
\end{proof}

\begin{corollary}
Let $f\in\CVF$ and let $v_{i}\in\partial f(x_{i})$ for $i=0,1$.
If $\exists x\in \left]x_{0},x_{1}\right[$ such that $\partial f(x)=\varnothing$, then
$\scal{v_{0}-v_{1}}{x_{0}-x_{1}}>0$.
\end{corollary}

\begin{lemma}\label{l:almoststrict} Suppose that $f\in\CVF$ is almost strictly convex. Then
$\partial f$ is strictly monotone on convex subsets of $\dom \partial f$.
\end{lemma}

\begin{proof} Suppose that $\partial f$ is not strictly monotone on convex subsets of $\dom \partial f$.
Then $\exists [x_{0},x_{1}]\subset\dom\partial f$ with $x_{0}\neq x_{1}$, $v_{i}\in\partial f(x_{i})$
such that $\scal{v_{0}-v_{1}}{x_{0}-x_{1}}=0$. By Lemma~\ref{l:not:strict}\ref{i:faffine},
$f$ is affine
on $[x_{0},x_{1}]$, which contradicts that $f$ is strictly convex on $[x_{0},x_{1}]$.
\end{proof}

\begin{lemma}\label{l:strictmono} Suppose that $f\in\CVF$ and that
$\partial f$ is strictly monotone on convex subsets of $\dom \partial f$. Then $f$ is almost strictly convex.
\end{lemma}
\begin{proof}
If $f$ is not almost strictly convex, then
$\exists [x_{0},x_{1}]\subset\dom\partial f$ such that $x_0\neq x_{1}$ and $f$ is affine
on $[x_{0},x_{1}]$. Because $[x_{0},x_{1}]\subset\dom\partial f$,
we have $\partial f(\barx)\neq \varnothing$ for every
$\barx\in \left]x_{0},x_{1}\right[$.
Take $v\in \partial f(\barx)$, and
apply Lemma~\ref{l:affine:segment} to obtain $v\in\partial f(x_{0})\cap\partial f(x_{1})$.
This contradicts that $\partial f$ is strictly monotone on $[x_{0},x_{1}]\subset\dom\partial f$.
\end{proof}

We are now ready for the

\noindent{\sl Proof of Theorem~\ref{t:almost}.}

``\ref{i:falmost}$\Leftrightarrow$\ref{i:almost}'':
Combine Lemmas~\ref{l:almoststrict} and \ref{l:strictmono}.

``\ref{i:almost}$\Rightarrow$\ref{i:full}'': Apply Lemma~\ref{l:not:strict}\ref{i:insubdomain}.

``\ref{i:full}$\Rightarrow$\ref{i:almost}'': Clear.
\qed

\begin{corollary}
Let $f\in\CVF$. Then the following are equivalent:
\begin{enumerate}
\item\label{i:friday1} $f$ is almost strictly convex.
\item\label{i:friday2}
 For all $x_{0},x_{1}\in \dom\partial f$, $x_{0}\neq x_{1}$, $v\in\partial f(x_{0})$,
we have
$f(x_{1})>f(x_{0})+\scal{v}{x_{1}-x_{0}}.$
\end{enumerate}
\end{corollary}
\begin{proof}
``\ref{i:friday1}$\Rightarrow$\ref{i:friday2}'':
We prove by contradiction. Suppose
$\exists x_{0}, x_{1}\in\dom\partial f$, $x_{0}\neq x_{1}$, and $v_{0}\in\partial f(x_{0})$
such that
\begin{equation}\label{e:affine}
f(x_{1})=f(x_{0})+\scal{v_{0}}{x_{1}-x_{0}}.
\end{equation}
Then
$f(x_{1})=f(x_0)-\scal{v_{0}}{x_{0}}+\scal{v_{0}}{x_{1}}=-f^*(v_{0})+\scal{v_{0}}{x_{1}}$
gives $f(x_{1})+f^*(v_{0})=\scal{v_{0}}{x_{1}}$, so $v_{0}\in\partial f(x_{1})$. Since
$\partial f$ is maximally monotone, $(\partial f)^{-1}$ is maximally monotone. This implies that
$(\partial f)^{-1}(v)$ is a closed convex set for every $v\in\HH$.
From $x_{0}, x_{1}\in (\partial f)^{-1}(v_{0})$, we deduce that
$[x_{0}, x_{1}]\in (\partial f)^{-1}(v_{0})$, in particular,
\begin{equation}\label{e:domain:line}
[x_{0},x_{1}]\subset\dom \partial f.
\end{equation}
In view of \eqref{e:affine} and \eqref{e:domain:line}, $f$ is not almost convex by Theorem~\ref{t:inequality:s},
which is a contradiction. Or Apply Theorem~\ref{t:almost}\ref{i:almost}.

``\ref{i:friday2}$\Rightarrow$\ref{i:friday1}'': Let $v_{i}\in\partial f(x_{i})$ with $i=0,1$ with $x_0\neq x_{1}$.
The assumption gives
\begin{equation*}
f(x_{1})>f(x_{0})+\scal{v_{0}}{x_{1}-x_{0}}, \text{ and  } f(x_{0})>f(x_{1})+\scal{v_{1}}{x_{0}-x_{1}}.
\end{equation*}
Adding these two inequalities yields
$\scal{v_{1}-v_{0}}{x_{1}-x_{0}}>0$. Thus, $\partial f$ is strictly monotone. Hence $f$ is almost strictly convex by
Theorem~\ref{t:almost}.
\end{proof}

The following result shows that under the assumption $f$ being subdifferentiable on $\dom f$,
Theorem~\ref{l:strictly:mono} is necessary and sufficient.
\begin{corollary}\label{c:strict:cm} Suppose that $f\in\CVF$ and $\dom f=\dom\partial f$. Then
$f$ is strictly convex if and only if $\partial f$ is strictly monotone.
\end{corollary}
\begin{proof}
Because $\dom f=\dom\partial f$, we see that $\dom\partial f$ is convex. Apply Theorem~\ref{t:almost}.
\end{proof}

\begin{lemma}\label{l:domandsub} Let $f\in\CVF$. Suppose that one of the following holds:
\begin{enumerate}
\item\label{i:open} $\dom f$ is open.
\item\label{i:closed} $\dom\partial f$ is closed.
\end{enumerate}
Then $\dom f=\dom\partial f$.
\end{lemma}

\begin{proof}
\ref{i:open}: Because $\dom f$ is open, $f$ is a continuous convex function. We have $\dom f\subseteq\dom\partial f$,
see, e.g., \cite[Proposition 16.27]{Bauschke-Combettes-2017}.
Since $\dom\partial f\subseteq \dom f$ always holds, the result follows.

\ref{i:closed}: By the Brondsted-Rockafellar theorem, see, e.g., \cite[Theorem 16.58]{Bauschke-Combettes-2017},
this follows from
$$\dom \partial f\subseteq \dom f\subseteq \closu{\dom\partial f}=\dom\partial f.$$
\end{proof}

The following significantly extends Theorem~\ref{l:strictly:mono}.
\begin{theorem}\label{t:open:closed} Let $f\in\CVF$.  Suppose that one of the following holds:
\begin{enumerate}
\item\label{i:open} $\dom f$ is open.
\item $\dom\partial f$ is closed.
\end{enumerate}
Then
$f$ is strictly convex if and only if $\partial f$ is strictly monotone.
\end{theorem}
\begin{proof}
Apply Lemma~\ref{l:domandsub} and Theorem~\ref{t:almost} or Corollary~\ref{c:strict:cm}.
\end{proof}

We end this section with some calculus on almost strictly convex functions, which
potentially allow us to construct more almost strictly convex functions.

\begin{theorem} Let $f_{1}, f_{2}\in\CVF$. Suppose that one of them is almost strictly
convex, and
$\dom f_{1}\cap\inte\dom f_{2}\neq\varnothing$ (or $\inte\dom f_{1}\cap\dom f_{2}\neq\varnothing$).
Then $f_{1}+f_{2}$ is almost strictly
convex.
\end{theorem}
\begin{proof} Suppose, without loss of generality,
that $f_{1}$ is almost strictly convex. Then $\partial f_{1}$ is strictly monotone
by Theorem~\ref{t:almost}.
Under the assumption, we have $\partial(f_{1}+f_{2})=\partial f_{1}+\partial f_{2}$ by
\cite[Theorem 2.8.7]{zalinescu} or \cite[Corollary~16.48]{Bauschke-Combettes-2017}.
Since $\partial f_{2}$ is monotone, $\partial(f_{1}+f_{2})$ is strictly monotone.
so Theorem~\ref{t:almost} applies.
\end{proof}

\begin{theorem} Let $f\in\CVF$ be almost strictly convex, and let
$A:\HH\rightarrow\HH$ be a bounded linear operator, injective and $\ran A\cap\inte(\dom f)\neq \varnothing$.
Then $f\circ A$ is almost strictly convex.
\end{theorem}
\begin{proof} Put $g:=f\circ A$.
By \cite[Theorem 16.47]{Bauschke-Combettes-2017} or \cite[Theorem 2.8.3]{zalinescu},
we have $\partial g(x) =A^{*}(\partial f(Ax))$. Let $x_{0}\neq x_{1}$ and $v_{i}\in\partial g(x_{i})$
with $i=0,1$. We have that $v_{i}=A^*x^*_{i}$ with $x^*_{i}\in\partial f(Ax_{i})$, and that
 $Ax_{0}\neq Ax_{1}$ because $A$ is injective. It follows that
\begin{align*}
\scal{v_{0}-v_{1}}{x_{0}-x_{1}} & =\scal{A^*x_{0}^*-A^*x_{1}^*}{x_{0}-x_{1}}\\
&=\scal{x_{0}^*-x_{1}^*}{Ax_{0}-Ax_{1}}>0,
\end{align*}
because $\partial f$ is strictly monotone by Theorem~\ref{t:almost}.
Hence $\partial g$ is strictly monotone, and Theorem~\ref{t:almost} applies again.
\end{proof}

\section{Paramonotone operators}\label{s:paramonotone}
In this section, we extend results on subdifferentials of convex functions
in Section~\ref{s:almostfunction} to paramonotone
operators which are maximally monotone.

To investigate almost strictly monotone operators, we shall need the following powerful notion
coined by Censor, Iusem, and Zenios \cite{censor}.
\begin{definition} An operator $A:\HH\rightarrow 2^{\HH}$ is paramonotone if
it is monotone and
$$[(\forall (x_{i},v_{i})\in\gra A, i=0,1)\ \scal{x_{0}-x_{1}}{v_{0}-v_{1}}=0] \Rightarrow
(x_{0}, v_{1}) \in \gra A, (x_{1},v_{0})\in\gra A.$$
\end{definition}

See \cite{BBHM,Bauschke-Combettes-2017,yao14,burachik,censor,iusem} for further details and applications in
optimization about paramontone operators.

\begin{definition} A monotone operator $A:\HH\rightarrow 2^{\HH}$ is at most single-valued if
$Ax$ is a singleton
for every $x\in\dom A$, but $Ax=\varnothing$ otherwise.
$A$ is said to be one-to-one if both $A$ and $A^{-1}$ are at most single-valued.
\end{definition}

The following two results on paramontone operators are essential.

\begin{lemma}\label{l:paramono:1} Let $A:\HH\rightarrow 2^{\HH}$ be paramonotone and maximally monotone.
Suppose that $\exists x_{0},x_{1}\in\dom A$, $x_{0}\neq x_{1}$, and $\exists v_{i}\in Ax_{i}$ such that
$\scal{x_{0}-x_{1}}{v_{0}-v_{1}}=0$.
Then $$(\forall x\in \left]x_{0},x_{1}\right[)\
Ax=Ax_{0}\cap Ax_{1}\supset [v_0, v_{1}].$$
\end{lemma}
\begin{proof} The paramonotonicity implies that $v_{0}, v_{1}\in Ax_{0}\cap Ax_{1}$. Because $A$ is maximally
monotone, each $Ax_{i}$ is a closed convex set, so $[v_{0},v_{1}]\subset Ax_{0}\cap Ax_{1}$. Now
$x_{0},x_{1}\in A^{-1}v_{0}$. Because $A^{-1}$ is maximally monotone, $A^{-1}v_{0}$ is
convex, so $[x_{0},x_{1}]\subset A^{-1}v_{0}$. Then $v_{0}\in Ax$ for $x\in [x_{0},x_{1}]$,
in particular, $[x_{0},x_{1}]\subset\dom A$.
By \cite[Proposition 22.10(ii)]{Bauschke-Combettes-2017}, applied to $x_{0},x_{1}$
with $Ax_{0}\cap Ax_{1}\neq\varnothing$ and
$x\in \left]x_{0},x_{1}\right[$,
we deduce that
$(\forall x\in \left]x_{0},x_{1}\right[)\ Ax=Ax_{0}\cap Ax_{1}$.
\end{proof}

\begin{lemma}\label{l:strict:disjoint}
Let $A:\HH\rightarrow 2^{\HH}$ be paramonotone. Then
$A$ is strictly monotone if and only if $A$ is disjointly injective, i.e.,
$Ax\cap Ay=\varnothing$ if $x\neq y$.
\end{lemma}
\begin{proof} ``$\Rightarrow$": Suppose that $A$ is strictly monotone.
Let $x\neq y$. We show $Ax\cap Ay=\varnothing$. If not, then $\exists u\in Ax\cap Ay$,
so $\scal{u-u}{x-y}=0$, which contradicts that $A$ is strictly monotone.

``$\Leftarrow$": Suppose that $A$ is disjointly injective.
We show that $A$ is strictly monotone. If not, then $\exists x\neq y$ with $u\in Ax, v\in Ay$
such that $\scal{u-v}{x-y}=0$. Since $A$ is paramnotone, we have
$u\in Ay$ and $v\in Ax$, so $\{u,v\}\subset Ax\cap Ay$, which contradicts that $A$ is disjointly injective.
\end{proof}

A maximally monotone operator enjoys the following important property.
\begin{fact}\label{f:firm}\emph{(\cite[Proposition 23.10]{Bauschke-Combettes-2017})}
Let $A:\HH\rightarrow 2^{\HH}$ be maximally monotone. Then
$J_{A}$ is firmly nonexpansive and $\dom J_{A}=\HH$, namely,
$$(\forall x,y\in\HH)\ \scal{J_{A}x-J_{A}y}{x-y}\geq \|J_{A}x-J_{A}y\|^2.$$
\end{fact}
See \cite{Bauschke-Combettes-2017, reich1, BeckSIAM, reich} for rich theory and abundant applications of resolvents and their variants. The main result of this section comes as follows.
\begin{theorem}\label{t:para} Let $A:\HH\rightarrow 2^{\HH}$ be paramonotone and maximally monotone.
Then the following are equivalent:
\begin{enumerate}
\item\label{i:m:strict} $A$ is strictly monotone.
\item\label{i:alm:strict} $A$ is almost strictly monotone.
\item\label{i:inverse} $A^{-1}$ is at most single-valued.
\item\label{i:resolvent} $J_{A}$ is strictly nonexpansive, i.e.,
$$(\forall x,y\in\HH, x\neq y)\ \|J_{A}x-J_{A}y\|<\|x-y\|.$$
\item\label{i:minus:resolvent} $\Id-J_{A}$ is injective.
\item\label{i:strict:iresovent} $\Id-J_{A}$ is strictly monotone.
\end{enumerate}
\end{theorem}
\begin{proof}
``\ref{i:m:strict}$\Rightarrow$\ref{i:alm:strict}'': Clear.

``\ref{i:alm:strict}$\Rightarrow$\ref{i:m:strict}'':
We prove by contradiction. Suppose that
$A$ is not strictly monotone. Then $\exists (x_{i},v_{i})\in\gra A, i=0,1$ such that $x_{0}\neq x_{1}$ and
$\scal{x_{0}-x_{1}}{v_{0}-v_{1}}=0$.
%Since $A$ is paramontone, we have
%$(x_{0},v_{1})\in\gra A, (x_{1},v_{0})\in\gra A$, thus,
%$v_{0}\in Ax_{0}\cap Ax_{1}, v_{1}\in Ax_{0}\cap Ax_{1}$. Since $x_{0}, x_{1}\in A^{-1}v_{0}$ and
%$A^{-1}$ is maximally monotone, $A^{-1}v_{0}$ is convex, so $[x_{0},x_{1}]\subset A^{-1}v_{0}$. Then
%$[x_{0},x_{1}]\subset\dom A$, and $v_{0}\in Ax$ for every $x\in [x_{0},x_{1}]$. In fact, by
%\cite[Proposition 22.10(ii)]{Bauschke-Combettes-2017},
It follows from Lemma~\ref{l:paramono:1} that
$Ax=Ax_0\cap Ax_{1}\supset\{v_{0},v_{1}\}$
for every $x\in \left]x_{0},x_{1}\right[$.
Then $A$ is not strictly monotone on $[x_{0},x_{1}]$.
Because $[x_{0},x_{1}]\subset\dom A$,
the assumption says that $A$ is strictly monotone on $[x_0,x_{1}]$. This is a contradiction.

``\ref{i:m:strict}$\Rightarrow$\ref{i:inverse}'': Suppose that $A$ is strictly monotone.
If $x_{0},x_{1}\in A^{-1}v$ and $x_{0}\neq x_{1}$, then
$v\in Ax_{0}\cap Ax_{1}$, implying $\scal{v-v}{x_{0}-x_{1}}=0$ with $x_{0}\neq x_{1}$. This contradicts
that $A$ is strictly monotone.

``\ref{i:inverse}$\Rightarrow$\ref{i:m:strict}'': Suppose that $A^{-1}$ is at most single-valued.
If $A$ is not strictly monotone, then $\exists v_{i}\in Ax_{i}$ for $i=0,1$ such that
$x_{0}\neq x_{1}$ and $\scal{v_{0}-v_{1}}{x_{0}-x_{1}}=0$. Since
$A$ is paramonotone, we have that $v_{0}\in Ax_{1}, v_{1}\in Ax_{0}$. Then
$\{x_{0}, x_{1}\}\subset A^{-1}v_{0}$, which contradict that $A^{-1}$ is at most single-valued.

``\ref{i:m:strict}$\Leftrightarrow$\ref{i:resolvent}'': This follows from
\cite[Theorem 2.1(ix)]{moffat12}, because of that under the assumption of paramonotonicity,
$A$ being strictly monotone is equivalent to $A$ being disjointly injective by Lemma~\ref{l:strict:disjoint}.
%To see this,
%suppose $A$ is strictly monotone. If $\exists x, y$ such that $x\neq y$ and $Ax\cap Ay\neq\varnothing$, we take $v\in Ax\cap Ay$ to get
%$\scal{v-v}{x-y}=0$, contradicting the strict monotonicity of $A$.
%Conversely, suppose $A$ is disjointly injective. If $\exists x, y$ such that
%$x\neq y$ and $u\in Ax, v\in Ay$ satisfying
%$\scal{u-v}{x-y}=0$, then $u\in Ax\cap Ay$ by the paramonotonicity of $A$, contradicting the disjoint injectivity
%of $A$.

``\ref{i:resolvent}$\Leftrightarrow$\ref{i:minus:resolvent}'': This is
\cite[Theorem 3.3(i)$\Leftrightarrow$(iii)]{moffat12}.

``\ref{i:resolvent}$\Leftrightarrow$\ref{i:strict:iresovent}'': Suppose that $J_{A}$ is strictly nonexpansive.
For $x\neq y$, we have
\begin{align*}
& \scal{x-y}{(x-J_{A}x)-(y-J_{A}y)}\\
&=\|x-y\|^2-\scal{x-y}{J_{A}x-J_{A}y}\\
&\geq \|x-y\|^2-\|x-y\|\|J_{A}x-J_{A}y\|\\
& >\|x-y\|^2-\|x-y\|^2=0,
\end{align*}
so $\Id-J_{A}$ is strictly monotone.
Conversely, suppose that $\Id-J_{A}$ is strictly monotone. That is,
$$(\forall x, y\in\HH, x\neq y)\ \scal{x-y}{(x-J_{A}x)-(y-J_{A}y)}>0.$$
Since $J_{A}$ is firmly nonexpansive by Fact~\ref{f:firm},
we have
\begin{align*}
\|x-y\|^2 &> \scal{x-y}{J_{A}x-J_{A}y}\\
&\geq \|J_{A}x-J_{A}y\|^2,
\end{align*}
so that $\|x-y\|>\|J_{A}x-J_{A}y\|$. Hence $J_{A}$ is strictly nonexpansive.
\end{proof}

\begin{corollary} Let $A:\HH\rightarrow 2^{\HH}$ be paramonotone and maximally monotone.
Then the following are equivalent:
\begin{enumerate}
\item\label{i:thurs1} $A$ is strictly monotone and at most single-valued.
\item\label{i:thurs2} $A^{-1}$ is strictly monotone and at most single-valued.
\item\label{i:thurs3} Both $J_{A}$ and $J_{A^{-1}}$ are is strictly nonexpansive.
\item\label{i:thurs4} $A$ is one-to-one.
\end{enumerate}
\end{corollary}
\begin{proof}
``\ref{i:thurs1}$\Leftrightarrow$\ref{i:thurs2}$\Leftrightarrow$\ref{i:thurs3}'':
 Observe that $A$ is paramonotone (maximally monotone) if and only if $A^{-1}$ is
paramonotone (maximally monotone). Apply Theorem~\ref{t:para} to both $A$ and $A^{-1}$.

``\ref{i:thurs1}$\Leftrightarrow$\ref{i:thurs4}'': Apply Lemma~\ref{l:strict:disjoint}.
\end{proof}

The following example shows that Theorem~\ref{t:para} fails if $A$ is not maximally monotone.

\begin{example} Define $$A:\RR\rightarrow 2^{\RR}:
x\mapsto \begin{cases}
-x^2, & \text{ if $x\leq 0$;}\\
\varnothing, & \text{ if $0<x<1$;}\\
(x-1)^2, & \text{ if $x\geq 1$.}
\end{cases}
$$
Then $A$ is paramonotone, but not maximally monotone.
Now $\dom A=\left]-\infty, 0\right]\cup \left[1, +\infty\right[$,
$A$ is strictly monotone on convex subsets
of $\dom A$, namely
$\left]-\infty, 0\right]$ and $\left[1, +\infty\right[$.
However, $A$ is not strictly monotone.
\end{example}

The next example shows that Theorem~\ref{t:para} fails if $A$ is not paramonotone.

\begin{example} Define the skew operator $$A:\RR^2\rightarrow\RR^2: (x_{1},x_{2})\mapsto
(-x_{2},x_{1}).
$$
Then $A$ is not paramonotone, but maximally monotone. We have
$A^{-1}$ is single-valued, but $A$ is not strictly monotone.
\end{example}

\section{Duality}\label{s:duality}
Duality has long been central in convex analysis;
see \cite{bb-combettes, Bauschke-Combettes-2017, Rock70, RW} for further details and references therein.
This section focuses on the characterizations of almost strictly convex functions by their Fenchel
conjugates. It turns out that seemingly different variants of almost differentiable (or essentially smooth)
 functions coincide if the domain of
its Fenchel conjugate has a nonempty interior.
In addition, based on duality, we show that a function is almost strictly convex if and only if its Moreau
envelope is strictly convex.
\begin{definition}\label{d:almost} A function $f\in \CVF$ is almost differentiable if
$\partial f(x)$ is either a singleton or an empty set for every $x\in\HH$. That is,
the mapping $\partial f$ is single-valued on its domain.
\end{definition}

In the light of Definition~\ref{d:almost}, the following observations help.
\begin{remark} \
\begin{enumerate}
\item Different from \cite[Definition 5.2(i)]{bb-combettes}, we do not require $\partial f$ to be locally bounded,
in particular, $\inte\dom f\neq\varnothing$. See the almost differentiable function $f^*$ given in
Example~\ref{e:empty:interior}.
%Different from \cite[Theorem 5.4]{bb-combettes}, we do not require $\partial f^*$ locally bounded,
%in particular, $\inte\dom f^*\neq \varnothing$.
\item In \cite[Theorem 11.13]{RW} Rockafellar and Wets originally define that $f\in\CVF$ is almost differentiable if $f$ is differentiable
on the open convex set $\inte(\dom f)$, $\inte(\dom f)\neq\varnothing$,
and $\partial f(x)=\varnothing$
for all points $x\in\RR^n\setminus(\inte\dom f)$, if any.
\end{enumerate}
\end{remark}
We will show that all these variants of almost differentiable functions in fact coincide if the function has a nonempty
interior domain!

Now, getting back to the course, almost strict convexity of a convex function dualizes to
almost differentiability of its Fenchel conjugate.

\begin{theorem}\label{t:almost:diff} Let $f\in \CVF$. Then the following are equivalent:
\begin{enumerate}
\item $f$ is almost strictly convex.
\item $f^*$ is almost differentiable.
\end{enumerate}
\end{theorem}
\begin{proof}
Since that $\partial f$ is paramonotone \cite[Example 22.4(i)]{Bauschke-Combettes-2017} and maximally monotone by \cite[Theorem 20.25]{Bauschke-Combettes-2017}, and that
$\partial f^*=(\partial f)^{-1}$, we
apply Theorem~\ref{t:para}\ref{i:m:strict}$\Leftrightarrow$\ref{i:inverse} and
Theorem~\ref{t:almost}.
\end{proof}

\begin{corollary} Let $f\in\CVF$ and $\dom\partial f=\inte\dom f$. Then
the following are equivalent:
\begin{enumerate}
\item\label{i:strict:interior} $f$
is strictly convex on the open convex set $\inte\dom f$.
\item $\partial f$ is strictly monotone.
\item\label{i:dual} $f^*$ is almost differentiable, in which case, $\dom \partial f^*=\partial f(\inte\dom f)$.
\end{enumerate}
\end{corollary}
\begin{proof} Observe that $\inte\dom f$ is convex.
Under the assumption $\dom\partial f=\inte\dom f$,
\ref{i:strict:interior} means that $f$ is almost strictly convex.
Hence Theorems~\ref{t:almost} and \ref{t:almost:diff} apply.

Finally, the assumption $\dom\partial f=\inte\dom f$ yields
$$\dom \partial f^*=\ran\partial f=\partial f(\dom\partial f)=\partial f(\inte\dom f).$$
\end{proof}

Two significant consequences of Theorem~\ref{t:almost:diff} come as follows.
First, it extends \cite[Theorem~11.13]{RW} from $\RR^n$ to a general Hilbert space.
This relies on the following crucial observation in $\RR^n$.
\begin{lemma}\label{l:finite} Let $f$ be a proper, lsc, and convex function
on $\RR^n$. Then $\partial f$ being at most single-valued is the
same as that $f$ being differentiable on the open convex set $\inte\dom f$, $\inte\dom f\neq \varnothing$, and $\partial f(x)=\varnothing$
for $x\in\RR^n\setminus(\inte\dom f)$.
\end{lemma}
\begin{proof}
Indeed,
if $\inte\dom f= \varnothing$, then $\dom f$ lies in a hyperplane, so that
$N_{\dom f}(x)\setminus\{0\}\neq\varnothing$ for every $x\in\dom f$. Take a point $x\in\dom\partial f$.
Since $N_{\dom f}(x)$ is a cone and
$\partial f(x)+N_{\dom f}(x)\subset\partial f(x)$, this contradicts that
$\partial f(x)$ is a singleton. Therefore, $\inte\dom f\neq\varnothing$.
For $x\in\inte\dom f$, $f$ is differentiable at $x$ by \cite[Theorem 17.18, Corollary 8.39]{Bauschke-Combettes-2017}
because $\partial f(x)$ is a singleton.
For $x\in\dom f\setminus(\inte\dom f)$ (if any), it is a boundary point, so $N_{\dom f}(x)\setminus\{0\}\neq\varnothing$.
If $\partial f(x) \neq\varnothing$, then $\partial f(x)+N_{\dom f}(x)\subset\partial f(x)$
implies that
$\partial f(x)$ is not a singleton, which is a contradiction. Therefore,
$\partial f(x)=\varnothing$ for every $x\in\dom f\setminus(\inte\dom f)$. Finally,
$\partial f(x)=\varnothing$ for $x\not\in\dom f$ by the definition.
\end{proof}

\begin{corollary}[Rockafellar-Wets]\label{c:rockwet}
Let $f\in \Gamma_{0}(\RR^n)$. Then the following are equivalent:
\begin{enumerate}
\item $f$ is almost strictly convex.
\item $f^*$ is almost differentiable, in the sense that $f^*$ is differentiable on the open convex
set $\inte\dom f^*$, which is nonempty, but $\partial f^*(x)=\varnothing$ for all points
$x\in\RR^n\setminus(\inte\dom f^*)$, if any.
\end{enumerate}
\end{corollary}

\begin{proof} Combine Theorem~\ref{t:almost:diff} and Lemma~\ref{l:finite}.
\end{proof}

Second, Theorem~\ref{t:almost:diff} not only provides new characterizations of almost strictly
convex functions via duality, but also recovers \cite[Corollary 18.11]{Bauschke-Combettes-2017}
with a different proof. To proceed, we need two important facts.

\begin{fact}\emph{(\cite[Propositions 6.45, 7.5]{Bauschke-Combettes-2017}, \cite[Theorem 1.1.3]{zalinescu})}
\label{f:nonempty:interior}
Let $C$ be a convex subset of $\HH$ such that $\inte C\neq\varnothing$ and $x\in C$. Then
$x\in\bdry C$ if and only if $N_{C}(x)\setminus\{0\}\neq \varnothing$.
\end{fact}

\begin{fact}\emph{(\cite[Theorem 3.3]{thibault2})}\label{f:thibault}
Let $f\in \CVF$ and $\inte\dom f\neq\varnothing$. Let $D$ be any dense subset of $\inte\dom f$.
Then for any $x\in\dom f$ one has
\begin{equation}\label{e:zagrodny}
\partial f(x)=\closu{\conv}^{w}\bigg(\limsup_{y\in D \rightarrow x}\partial f(y)\bigg)+N_{\dom f}(x)
=\closu{\conv}\bigg(\limsup_{y\in D \rightarrow x}\partial f(y)\bigg)+N_{\dom f}(x),
\end{equation}
where
$$\limsup_{y\in D, y\rightarrow x}\partial f(y)=\menge{y^*\in\HH}{y_{i}^*\rightharpoonup y^*, y_{i}^*\in\partial f(y_{i}), y_{i}\in D\rightarrow x},$$
and $N_{\dom f}$ denotes the normal cone to $\dom f$.
\end{fact}
\begin{proof} In \eqref{e:zagrodny}, while
the first equality is \cite[Theorem 3.3]{thibault2},
the second equality follows from \cite[Theorem 3.34]{Bauschke-Combettes-2017}.
\end{proof}

The following auxiliary result says that almost differentiable functions in various disguises are in fact
the same, if the function has a nonempty interior domain.
\begin{lemma}\label{l:infinite} Let $f\in \CVF$ and $\inte\dom f\neq\varnothing$. Then
the following are equivalent:
\begin{enumerate}
\item\label{i:sin1} $\partial f$ is at most single-valued.
\item\label{i:sin2}  $f$ is G\^ateaux differentiable on the open convex set $\inte\dom f$, and $\partial f(x)=\varnothing$
for $x\in\HH\setminus(\inte\dom f)$.
\item\label{i:sin3}
 $f$ is G\^ateaux differentiable on the open convex set $\inte\dom f$, and $\lim_{i\rightarrow\infty}\|\nabla f(x_{i})\|
=+\infty$ whenever $(x_{i})_{i\in\NN}$ is a sequence in $\inte\dom f$ converging to a boundary point of
$\inte\dom f$.
\end{enumerate}
\end{lemma}
\begin{proof}
``\ref{i:sin1}$\Rightarrow$\ref{i:sin2}": Suppose that $\partial f$ is at most single-valued.
For $x\in\inte\dom f$, $f$ is differentiable at $x$ by \cite[Theorem 17.18, Corollary 8.39]{Bauschke-Combettes-2017}
because $\partial f(x)$ is a singleton.
For $x\in\dom f\setminus(\inte\dom f)$ (if any), it is a boundary point, so $N_{\dom f}(x)\setminus\{0\}\neq\varnothing$
by Fact~\ref{f:nonempty:interior}.
If $\partial f(x) \neq\varnothing$, by the subgradient inequality or \cite[Proposition 21.17]{Bauschke-Combettes-2017},
we have $\partial f(x)+N_{\dom f}(x)\subset\partial f(x)$
implying that
$\partial f(x)$ is not a singleton, which is a contradiction. Therefore,
$\partial f(x)=\varnothing$ for every $x\in\dom f\setminus(\inte\dom f)$. Finally,
$\partial f(x)=\varnothing$ for $x\not\in\dom f$ by the definition.

``\ref{i:sin2}$\Rightarrow$\ref{i:sin1}": Clear.

``\ref{i:sin2}$\Rightarrow$\ref{i:sin3}": Let $(x_{i})_{i\in\NN}$ be a sequence in $\inte\dom f$ and
$x_{i}\rightarrow x$ with $x$ being a boundary point of $\inte\dom f$. If $(\|\nabla f(x_{i})\|)_{i\in\NN}$
does not goes to $\infty$, it must have a bounded subsequence. Without loss of generality, we can assume
$(\|\nabla f(x_{i})\|)_{i\in\NN}$ is bounded. Then there exists a subsequence $(\nabla f(x_{i_{k}}))_{k\in\NN}$
weakly convergent to some $v\in\HH$. Since
$$(\forall y\in\HH)\ f(y)\geq f(x_{i_{k}})+\scal{\nabla f(x_{i_{k}})}{y-x_{i_{k}}},$$
taking $\liminf$ as $k\rightarrow\infty$ both sides we get
$$(\forall y\in\HH)\ f(y)\geq f(x)+\scal{v}{y-x}$$
so that $v\in\partial f(x)$. This contradicts that $\partial f(x)=\varnothing$.

``\ref{i:sin3}$\Rightarrow$\ref{i:sin2}": Since $\partial f(x)=\varnothing$ for $x\not\in\dom f$,
we only need to show $\partial f(x)=\varnothing$ for $x\in\dom f\setminus(\inte\dom f)$.
Let $x\in\dom f\setminus(\inte\dom f)$ and $\partial f(x)\neq \varnothing$. Since
$\partial f(x)=\closu{\conv}^{w}\big(\limsup_{y\in\inte\dom f, y\rightarrow x}\partial f(y)\big)
+N_{\dom f}(x)$ by Fact~\ref{f:thibault}, we have
$\limsup_{y\in\inte\dom f, y\rightarrow x}\partial f(y)\neq \varnothing$. Then
there exists a sequence $(x_{i})_{i\in\NN}$ in $\inte\dom f$ such that $x_{i}\rightarrow x$ and $(\nabla f(x_{i}))_{i\in\NN}$
is weakly convergent. Thus $x_{i}\rightarrow x$ and $(\|\nabla f(x_{i})\|)_{i\in\NN}$ is bounded,
which is a contradiction.
\end{proof}

\begin{theorem}\label{t:nonempty:int} Let $f\in \CVF$ and $\inte\dom f\neq\varnothing$.
Then the following are equivalent:
\begin{enumerate}
\item\label{i:Ga} $f$ is G\^ateaux differentiable on the open convex
set on $\inte\dom f$ and $\partial f(x)=\varnothing$ for $x\in\HH\setminus(\inte\dom f)$.
\item\label{i:Al} $f^*$ is almost strictly convex, i.e., $f^*$ is strictly convex on convex subsets of $\dom\partial f^*$,
in which case $\dom \partial f^*=\nabla f(\inte\dom f)$.
\end{enumerate}
\end{theorem}
\begin{proof}
``\ref{i:Ga}$\Rightarrow$\ref{i:Al}'':
Since $\partial f$ is at most single-valued, $f$ is almost differentiable. Theorem~\ref{t:almost:diff} applies.

``\ref{i:Al}$\Rightarrow$\ref{i:Ga}'': Theorem~\ref{t:almost:diff} implies that $f$ is almost differentiable,
so $\partial f$ is at most single-valued. Lemma~\ref{l:infinite} applies.

The claim that $\dom \partial f^*=\nabla f(\inte\dom f)$ follows from
\begin{equation*}
\dom\partial f^*=\ran\partial f=\partial f(\dom\partial f)=\partial f(\inte\dom f)=\nabla f(\inte\dom f).
\end{equation*}
\end{proof}

Applying Theorem~\ref{t:nonempty:int} to $f^*$ we obtain an exact analogue of Corollary~\ref{c:rockwet}
in a general Hilbert space.
% comes as follows.
\begin{corollary}\label{c:dual:diff}
Let $f\in \CVF$ and $\inte\dom f^*\neq\varnothing$.
Then the following are equivalent:
\begin{enumerate}
\item\label{i:Alc} $f$ is almost strictly convex.
\item\label{i:Gac} $f^*$ is G\^ateaux differentiable on the open convex
set $\inte\dom f^*$, and $\partial f^*(x)=\varnothing$ for $x\in\HH\setminus(\inte\dom f^*)$.
\end{enumerate}
These conditions imply
$\dom\partial f=\nabla f^*(\inte\dom f^*) \text{ and }
\ran\partial f=\inte\dom f^*.$
\end{corollary}

Moreover, the following result \cite[Corollary 18.11]{Bauschke-Combettes-2017} is immediate.
\begin{corollary}
Let $f\in \CVF$ such that $\dom\partial f=\inte\dom f$. Then the following are equivalent:
\begin{enumerate}
\item\label{i:diff:int} $f$ is G\^ateaux differentiable on the open convex
set on $\inte\dom f$.
\item $f^*$ is almost strictly convex, i.e., $f^*$ is strictly convex on convex subsets of $\dom\partial f^*$,
in which case $\dom \partial f^*=\nabla f(\inte\dom f)$.
\end{enumerate}
\end{corollary}
\begin{proof}
Because $\dom\partial f\neq\varnothing$, the assumption
$\dom\partial f=\inte\dom f$ implies $\inte\dom f\neq \varnothing$. It remains to apply
Theorem~\ref{t:nonempty:int}.
%\ref{i:diff:int} means that
%$f$ is almost differentiable. Then Theorem~\ref{t:almost:diff} applies.
%The claim $\dom \partial f^*=\nabla f(\inte\dom f)$ follows from
%\begin{equation}
%\dom\partial f^*=\ran\partial f=\partial f(\dom\partial f)=\partial f(\inte\dom f)=\nabla f(\inte\dom f).
%\end{equation}
\end{proof}

\begin{remark}
Some comments are in order.
\begin{enumerate}
\item Let $f\in \CVF$ with $\inte\dom f\neq\varnothing$.
If $f\in\CVF$ is almost differentiable, then $f$ is G\^ateaux differentiable on $\inte\dom f$ by
\cite[Theorem 17.18, Corollary 8.39]{Bauschke-Combettes-2017}. Simple examples show that the converse fails, e.g.,
$$f:\RR\rightarrow\RX: x\mapsto \begin{cases}
x, &\text{ if $x\geq 0$;}\\
+\infty, &\text{ if $x<0$.}
\end{cases}
$$
\item Let $f\in\CVF$. Then $\dom\partial f$ is open if and only if $\dom\partial f=\inte\dom f$.
Indeed, on one hand
$\dom\partial f$ being open
implies $\dom\partial f\subset\inte\dom f$ because $\dom\partial f\subset\dom f$; on the other hand,
$\inte\dom f\subset \dom\partial f$ always hold, see,
e.g., \cite[Corollary~8.39, Proposition~16.17(iv)]{Bauschke-Combettes-2017}. Thus,
$\dom\partial f$ being open gives $\dom\partial f=\inte\dom f$. The other direction is clear.
\end{enumerate}
\end{remark}

The following result establishes a characterization of $\partial f$ being one-to-one.
\begin{corollary} Let $f\in\CVF$. Then $\partial f$ is one-to-one if and only if $f$ is almost strictly convex and
almost differentiable. If, in addition, $\inte\dom f\neq\varnothing$, then
$\partial f$ is one-to-one if and only if $f$ is strictly convex and G\^ateaux differentiable on $\inte\dom f$, and
$\partial f(x)=\varnothing$ for $x\in\HH\setminus(\inte\dom f)$.
\end{corollary}
\begin{proof} $\partial f$ being one-to-one means: $\partial f$ is at most single-valued, so $f$ is almost differentiable; and
%$\partial f$ is injective, so
$(\partial f)^{-1}$ is at most single-valued, so $f^*$ is almost differentiable.
The latter ensures that $f$ is almost strictly convex by Theorem~\ref{t:almost:diff}.

If $\inte\dom f\neq\varnothing$, by Lemma~\ref{l:infinite}, $f$ being almost differentiable is equivalent
to that $f$ is G\^ateaux differentiable on the open convex set $\inte\dom f$ and $\partial f(x)=\varnothing$ for
$x\in\HH\setminus (\inte\dom f)$. Thus, $f$ being almost strictly convex reduces to that
$f$ is strictly convex on $\inte\dom f$.
\end{proof}

%\begin{remark} Observe that $\dom\partial f=\inte\dom f$ if and only if $\dom\partial f$ is open.
%Indeed, since $\dom\partial f\subset\dom f$ and $\dom\partial f$ is open, it follows that
%$\dom\partial f\subset\inte\dom f\subset \dom\partial f,$
%as required.
%\end{remark}
We end this section with some results on Moreau envelopes of almost strictly convex functions.
For $f\in\CVF$ and parameter value $\lambda \in \left]0,+\infty\right[$,
the Moreau envelope $e_{\lambda}f$ and
proximal mapping $\prox{\lambda}f$
are defined by
$$e_{\lambda}f(x):=\inf_{w\in\HH}\left\{f(w)+\frac{1}{2\lambda}\|x-w\|^2\right\},$$
$$\prox{\lambda}f(x):=\argmin_{w\in\HH}\left\{f(w)+\frac{1}{2\lambda}\|x-w\|^2\right\}.$$
Set $\jj:=\tfrac{1}{2}\|\cdot\|^2$. The following result extends \cite[Theorem 3.7]{planiden19} from
a finite-dimensional space to a general Hilbert space.

\begin{theorem}\label{t:envel} Let $f\in\CVF$ and let the parameter
$\lambda\in \left]0,+\infty\right[$. Then the following are equivalent:
\begin{enumerate}
\item\label{i:f:alconvex} $f$ is almost
strictly convex.
\item\label{i:env:sconvex} $e_{\lambda}f$ is strictly convex.
\item \label{i:prox:strict}
$\prox{\lambda}f$ is strictly nonexpansive, i.e.,
$$(\forall x\neq y)\ \|\prox{\lambda}f(x)-\prox{\lambda}f(y)\|<\|x-y\|.$$
\end{enumerate}
\end{theorem}
\begin{proof}
``\ref{i:f:alconvex}$\Rightarrow$\ref{i:env:sconvex}'':
Assume that $f$ is almost strictly convex.
By Theorem~\ref{t:almost:diff}, $f^*$ is almost differentiable, i.e.,
$\partial f^*$ is at most single-valued.
Since
$(e_{\lambda}f)^*=f^*+\lambda \jj$, implying $\partial (e_{\lambda}f)^*=\partial f^*+\lambda\Id$,
we have that $(e_{\lambda}f)^*$ is almost differentiable. Applying Theorem~\ref{t:almost:diff}
again gives that $e_{\lambda}f$ is almost strictly convex.
Because $\nabla e_{\lambda}f=(\Id-\prox{\lambda}f)/\lambda$ has full domain, we deduce that
$e_{\lambda}f$ is strictly convex on $\HH$.

``\ref{i:env:sconvex}$\Rightarrow$\ref{i:f:alconvex}'': Assume that $e_{\lambda}f$ is strictly convex. Then
$e_{\lambda}f$ is almost strictly convex, so $(e_{\lambda}f)^*=f^*+\lambda\jj$ is almost differentiable
by Theorem~\ref{t:almost:diff}.
This implies that $f^*$ is almost differentiable, thus $f$ is almost strictly convex by Theorem~\ref{t:almost:diff}
again.

``\ref{i:env:sconvex}$\Rightarrow$\ref{i:prox:strict}'':
Because $\dom e_{\lambda}f =\HH$,
\ref{i:env:sconvex} is equivalent to $\nabla e_{\lambda}f$ being strictly monotone
by Theorem~\ref{t:open:closed}\ref{i:open}.
Then
\begin{equation*}
(\forall x\neq y)\ \scal{x-y}{\lambda^{-1}(x-\prox{\lambda}f(x))-\lambda^{-1}(y-\prox{\lambda}f(y))}>0,
\end{equation*}
which simplifies to
$$\|x-y\|^2>\scal{x-y}{\prox{\lambda}f(x)-\prox{\lambda}f(y)}.$$
Since $\prox{\lambda}f$ is firmly nonexpansive, we have
$$\scal{x-y}{\prox{\lambda}f(x)-\prox{\lambda}f(y)}\geq \|\prox{\lambda}f(x)-\prox{\lambda}f(y)\|^2.$$
Thus, $\|\prox{\lambda}f(x)-\prox{\lambda}f(y)\|<\|x-y\|$.

``\ref{i:prox:strict}$\Rightarrow$\ref{i:env:sconvex}'':
For $x\neq y$, the Cauchy-Schwarz inequality gives
\begin{align*}
\scal{x-y}{\prox{\lambda}f(x)-\prox{\lambda}f(y)} &\leq \|x-y\|\|\prox{\lambda}f(x)-\prox{\lambda}f(y)\|\\
&<\|x-y\|\|x-y\|=\|x-y\|^2,
\end{align*}
from which
\begin{align*}
& \scal{x-y}{\nabla e_{\lambda}f(x)-\nabla e_{\lambda}f(y)}\\
&=\scal{x-y}{\lambda^{-1}(x-\prox{\lambda}f(x))-\lambda^{-1}(y-\prox{\lambda}f(y))}>0,
\end{align*}
i.e., $\nabla e_{\lambda}f$ is strictly monotone. Hence $e_{\lambda}f$ is strictly convex by
Theorem~\ref{t:open:closed}.
\end{proof}

\begin{remark} Theorem~\ref{t:envel}\ref{i:f:alconvex}$\Leftrightarrow$\ref{i:prox:strict}
can also be obtained by using \cite[Theorem 2.1(ix)]{moffat12} and paramonotoncity of $\partial f$.
Theorem~\ref{t:envel} also extends \cite[Lemma 3.8]{moffat12} from a finite-dimensional space to
a general Hilbert space.
\end{remark}

Theorem~\ref{t:envel} can significantly simplify some proofs on
the proximal average \cite{bglw08} and
\cite[Section~14.2]{Bauschke-Combettes-2017}.
%Let $f_{1}, f_{2}\in\CVF$ and parameters $\lambda, \alpha>0$. Recall that the proximal average of $f_{1}, f_{2}$
%is defined by
%\begin{equation}\label{e:prox:def}
%f:=\big(\alpha (f_{1}+\lambda^{-1}\jj)*+(1-\alpha)(f_{2}+\lambda^{-1}\jj)^*\big)^*-\lambda^{-1}\jj.
%\end{equation}
\begin{corollary} Let $f_{1}, f_{2}\in\CVF$, let the parameter $\lambda \in \left]0,+\infty\right[$,
  and let the parameter $\alpha\in \left]0,1\right[$.
Define the proximal average of $f_{1}, f_{2}$ by
\begin{equation}\label{e:prox:def}
f:=\big(\alpha (f_{1}+\lambda^{-1}\jj)^*+(1-\alpha)(f_{2}+\lambda^{-1}\jj)^*\big)^*-\lambda^{-1}\jj.
\end{equation}
%$f$ be the proximal average given by \eqref{e:prox:def}.
Then the following hold:
\begin{enumerate}
\item\label{i:s:convex} If one of $f_{1}, f_{2}$
is almost strictly convex, then
$f$ is almost strictly convex.
\item\label{i:d:differ} If one of $f_{1}, f_{2}$ is almost differentiable, then
$f$ is almost differentiable.
\end{enumerate}
\end{corollary}
\begin{proof}
Write \eqref{e:prox:def} as
$(f+\lambda^{-1}\jj)^*=\alpha (f_{1}+\lambda^{-1}\jj)^*+(1-\alpha)(f_{2}+\lambda^{-1}\jj)^*$,
i.e.,
\begin{equation}\label{e:moreau:dual}
e_{\lambda^{-1}}f^*=\alpha e_{\lambda^{-1}}f_{1}^*+(1-\alpha)e_{\lambda^{-1}}f_{2}^*.
\end{equation}
Thus, \eqref{e:moreau:dual} and Moreau's decomposition \cite[Theorem 14.3(i)]{Bauschke-Combettes-2017}
 together yield
\begin{equation}\label{e:moreau:primal}
e_{\lambda}f=\alpha e_{\lambda}f_{1}+(1-\alpha)e_{\lambda}f_{2}.
\end{equation}

\ref{i:s:convex}: Suppose, without loss of generality, that $f_{1}$ is almost strictly convex.
Then
$e_{\lambda}f_{1}$ is strictly convex by Theorem~\ref{t:envel}. Since
\eqref{e:moreau:primal} holds,
%$e_{\lambda}f=\alpha e_{\lambda}f_{1}+(1-\alpha) e_{\lambda}f_{2}$,
we have that $e_{\lambda}f$ is strictly convex. Thus,
$f$ is almost strictly convex by Theorem~\ref{t:envel}.

\ref{i:d:differ}: Suppose, without loss of generality, that $f_{1}$ is almost differentiable.
Then $f_{1}^*$ is almost strictly convex
by Theorem~\ref{t:almost:diff}, so that
$e_{\lambda^{-1}}f_{1}^*$ is strictly convex by Theorem~\ref{t:envel}.
In view of \eqref{e:moreau:dual}
%$e_{\lambda}f^*=\alpha e_{\lambda}f_{1}^*+(1-\alpha) e_{\lambda}f_{2}^*$,
we have that $e_{\lambda^{-1}}f^*$ is strictly convex.
Then $f^*$ is almost strictly convex by Theorem~\ref{t:envel}, so that $f$ is almost differentiable by Theorem~\ref{t:almost:diff}.
\end{proof}

\section{Perspectives and open problems}\label{s:openprobs}
In this last section, we study almost strictly convex functions from
the tilted-stable optimization point of view
and present some open questions.
%versus
For every $f\in \CVF$ and $x^*\in\HH$, one has
\begin{equation}\label{e:tilted}
\partial f^*(x^*)=\argmin_{x\in\HH}\big(f(x)-\scal{x^*}{x}\big).
\end{equation}
Moreover, $\dom\partial f^*=\ran\partial f$ and $\ran\partial f^*=\dom\partial f$.
\eqref{e:tilted} says that
$\partial f^*(x^*)$ is the set of (global) minimizers of the function
$x\mapsto f(x)-\scal{x^*}{x}$, a tilted version of $f$.
The continuity of $x^*\mapsto \argmin_{x\in\HH}\big(f(x)-\scal{x^*}{x}\big)$ is the central topic
of tilted-stable minimization (local or global); see, e.g., \cite{nghia, poliquin, volle12}.
Observe that different differentiabilities
correspond to different continuities of the gradient mapping of a convex function; see, e.g.,
\cite[Section 17.6]{Bauschke-Combettes-2017} and \cite{phelps}. This leads to

\begin{lemma}\label{l:weak:c}
Let $f\in \CVF$ and $\inte\dom f^*\neq\varnothing$.
If $f^*$ is G\^ateaux differentiable on the open convex
set $\inte\dom f^*$, then
the mapping
$$x^*\mapsto \argmin_{x\in\HH}\big(f(x)-\scal{x^*}{x}\big)$$
is single-valued and strong-to-weak (i.e., norm-to-weak) continuous on $\inte\dom f^*$.
\end{lemma}
\begin{proof} Use \cite[Corollary 17.42]{Bauschke-Combettes-2017} (or \cite[Proposition 2.8]{phelps}) and \eqref{e:tilted}.
\end{proof}

\begin{theorem}
Let $f\in \CVF$ and $\inte\dom f^*\neq\varnothing$.
Then the following are equivalent:
\begin{enumerate}
\item\label{i:weak1} $f$ is almost strictly convex.
\item \label{i:weak2} The mapping
$$x^*\mapsto \argmin_{x\in\HH}\big(f(x)-\scal{x^*}{x}\big)$$
is single-valued and strong-to-weak (i.e., norm-to-weak) continuous on $\inte\dom f^*$. Moreover,
\begin{equation*}
(\forall x^*\in\HH\setminus(\inte\dom f^*))\ \argmin_{x\in\HH}\big(f(x)-\scal{x^*}{x}\big)=\varnothing.
\end{equation*}
\end{enumerate}
\end{theorem}
\begin{proof} ``\ref{i:weak1}$\Rightarrow$\ref{i:weak2}":
Combine Lemma~\ref{l:weak:c} and Corollary~\ref{c:dual:diff}.

``\ref{i:weak2}$\Rightarrow$\ref{i:weak1}'': Apply Corollary~\ref{c:dual:diff}.
\end{proof}

Furthermore, one has the following results.
\begin{lemma}\label{l:strong1}
Let $f\in \CVF$ and $\inte\dom f^*\neq\varnothing$.
If $f^*$ is Fr\'echet differentiable on the open convex
set $\inte\dom f^*$, then
the mapping
$$x^*\mapsto \argmin_{x\in\HH}\big(f(x)-\scal{x^*}{x}\big)$$
is single-valued and  continuous (i.e., norm-to-norm) on $\inte\dom f^*$.
\end{lemma}
\begin{proof}
Use \cite[Corollary 17.43]{Bauschke-Combettes-2017} (or \cite[Proposition 2.8]{phelps}) and \eqref{e:tilted}.
\end{proof}
\begin{lemma}\label{l:strong2}
Let $f\in \CVF$ and $\inte\dom f^*\neq\varnothing$.
Suppose that $f^*$ is Fr\'echet differentiable on the open convex set $\inte\dom f^*$ and that $\nabla f^*$ is locally (or globally) Lipschitz. Then
the mapping
$$x^*\mapsto \argmin_{x\in\HH}\big(f(x)-\scal{x^*}{x}\big)$$
is single-valued and locally (or globally) Lipschitz on $\inte\dom f^*$.
\end{lemma}
Of course, all these will lead to other stronger types of almost strictly convex functions.
We conclude this paper with an open problem.
 %\begin{enumerate}[label=\textbf{P\arabic*},ref=\arabic*,leftmargin=*]\item
    \begin{openprob}
    Characterize those types of almost strictly (or strongly) convex functions given by Lemmas~\ref{l:strong1} and
    \ref{l:strong2}.
    \end{openprob}

\section*{Acknowledgments}
HHB and XW were partially supported by NSERC Discovery Grants. HL was partially supported
by the Grants of NSF China and Chongqing (11991024, 11771064, KJZD-K 202500507).

\section*{Data availability statements}
The authors declare that the data supporting the findings of this study are available within the paper.

% REFERENCES------------------------------------------------------------------

%\appendixpage
%
%\begin{appendices}
%  \crefalias{section}{appsec}
%
%  \section{}\label{app:blowsup}
%
%  For the sake of completeness,
%  we provide the following proof of \cref{l:blowsup}
%  based on
%  \cite[Problem~3.2.43]{Kaczor.Nowak-1}.
%
%    \begin{proof}[Proof of \cref{l:blowsup}]
%    Because
%   ??
%
%    \end{proof}
%
%
%
%\end{appendices}

\end{document}